\documentclass[12pt]{article}

\usepackage{graphicx}
\usepackage{amssymb,amsmath,amsthm}

\newcommand{\maybeqed}{}

\numberwithin{equation}{section}

\renewcommand\setminus\smallsetminus 

\newcommand\C{{\mathbb{C}}}
\newcommand\Chat{{\hat{\mathbb{C}}}}

\newcommand\D{{\mathbb{D}}}
\newcommand\E{{\mathbf{E}}}
\newcommand\F{{\mathcal{F}}}
\newcommand\Q{{\mathbf{Q}}}
\newcommand\btilde{{\tilde{b}}}
\newcommand\cc{{\mathbf{c}}}
\newcommand\SLEk{{SLE$_\kappa$}}
\newcommand\eqd{\stackrel{\mathrm{d}}{=}}
\newcommand\sm\setminus
\newcommand\ol\overline
\newcommand\inv{^{-1}}
\newcommand\tmass\Psi
\newcommand\lm\Lambda
\newcommand\lmstar{{\lm^*}}
\newcommand\loopm{\mu^{\operatorname{loop}}}
\DeclareMathOperator{\ccap}{cap}
\DeclareMathOperator{\rad}{rad}

\theoremstyle{definition}
\newtheorem{definition}{Definition}[section]
\newtheorem{remark}[definition]{Remark}

\theoremstyle{plain}
\newtheorem{corollary}[definition]{Corollary}
\newtheorem{lemma}[definition]{Lemma}
\newtheorem{proposition}[definition]{Proposition}

\newtheorem{theorem}[definition]{Theorem}

\newcommand{\Prob} {{\bf P}}

\newcommand{\dist}{{\rm dist}}

\newcommand \hulls{{\mathcal H}}

\newcommand \p {\partial}

\newcommand \Half {{\mathbb H}}

\newcommand \Disk {{\mathbb D}}

\newcommand \domain {{\mathcal D}}

\newcommand \Outer {{\mathcal O}}
\newcommand \exc {{\mathcal E}}

\title{Reversed radial SLE\\and the Brownian loop measure}

\author{Laurence S. Field\thanks{University of Chicago.}\and
Gregory F. Lawler\thanks{University of Chicago; 
research supported by NSF grant DMS-0907143.}}

\date{February 28, 2013}

\begin{document}

\maketitle

\begin{abstract}
The Brownian loop measure is a conformally invariant
measure on loops in the plane that arises when studying the
Schramm--Loewner evolution (SLE).
When an SLE curve in a domain evolves from an interior
point, it is natural to consider the loops that hit the
curve and leave the domain, but their measure is infinite.
We show that there is a related normalized quantity that is
finite and invariant under M\"obius transformations of the
plane.
We estimate this quantity when the curve is small and the
domain simply connected.
We then use this estimate to prove a formula for the
Radon--Nikodym derivative of reversed radial SLE with
respect to whole-plane SLE.
\end{abstract}

\section{Introduction}
\label{sec:intro}

Oded Schramm introduced the Schramm--Loewner evolution
(SLE) in~\cite{Schramm} as a one-parameter 
family of random curves defined in
simply connected complex domains.
The parameter $\kappa>0$ determines the local
behaviour of the curve.
Schramm considered three types of SLE: \emph{chordal} SLE,
which connects two boundary points, \emph{radial} SLE,
which connects a boundary point to an interior point,
and \emph{whole-plane} SLE, which connects two points on
the Riemann sphere.
These random curves have two defining properties.
First, they are invariant under conformal transformations
of the domain.
Second, they satisfy the domain Markov property:
given an initial segment of the curve, the remainder of the
curve follows the law of SLE in the slit domain.
If $0 < \kappa \leq 4$, which is what we consider in this
paper, the curves are simple. 

As in \cite{LJSP,L11}, we will consider \SLEk\ in a domain $D$, 
when $ \kappa \leq 4$, as a measure $\mu_D(z,w)$ on
simple curves in $D$ connecting $z$ and $w$.  Here
$z,w$ can be interior or boundary points, but if
$z$ or $w$ is a 
 boundary point, we assume that the boundary is
locally analytic there.   The measures are conformally covariant;
that is, if $f:D \rightarrow f(D)$ is a conformal transformation,
then
\begin{equation}  \label{confconv}
   f \circ \mu_D(z,w) = |f'(z)|^{b_z} \, |f'(w)|^{b_w} \,
\mu_{f(D)}(f(z),f(w)), 
\end{equation}
where  $b_z$ and $b_w$ are the \emph{boundary scaling exponent}
\[  b := \frac {6-\kappa}{2\kappa}, \]
or 
the \emph{interior scaling exponent}
\[    \tilde b  := \frac{(\kappa - 2) \, b}{4}, \]
depending on which kind of point $z$ and $w$ are.   We
write the total mass of the measure $\mu_D(z,w)$ as
$\Psi_D(z,w)$ and call it the SLE \emph{partition function}.
See also Dub\'edat's work~\cite{Dub} for a slightly different notion of SLE partition function in the context of the Gaussian free field.

Indeed, in many examples one can obtain the SLE partition function (at least 
conjecturally) as a normalized limit of partition functions
of discrete measures.  It is known that if $\kappa \leq 8/3$,
or if $D$ is simply connected or doubly connected, then 
$\Psi_D(z,w) < \infty$.  It is conjectured that this is true
for all $D$ for $\kappa \leq 4$.    If $\Psi_D(z,w) < \infty$,
then we define $\mu^\#_D(z,w)$ to be the corresponding
probability measure obtained by normalizing.  The measure
$\mu_D^\#(z,w)$ is conformally \emph{invariant} and hence
can be defined for nonsmooth boundary points provided that
there exists a conformal transformation
$f:D \rightarrow f(D)$ such that $f(z),f(w)$ are smooth boundary
points and $\Psi_{f(D)}(f(z),f(w)) < \infty$. 

Suppose $z \in \partial D$ is a smooth boundary
point, and suppose that $ D_1$ is a subdomain of $D$ 
that agrees with it in a neighborhood of $z$.  Let us compare
$\mu:=\mu_D(z,w)$ with $ \mu_1 := \mu_{  D_1}(z,w_1)$.  Here
$w, w_1$ can be either boundary or interior points.  Let
$t$ be a stopping time for the SLE paths
 such that $\gamma_t :=
\gamma(0,t]$ lies in $  D_1$.  Then $\mu$
and $\mu_1$ considered as measures on initial segments
$\gamma_t$  are mutually absolutely continuous with
Radon-Nikodym derivative
\begin{equation}  \label{jun3.2}
   \frac{d\mu_1}{d\mu}\, (\gamma_t) =
    \frac{\Psi_{  D_1 \setminus \gamma_t}(\gamma(t),
   w_1)}{\Psi_{D \setminus \gamma_t}(\gamma(t),
      w)} \,  
\exp \Bigl\{ \frac\cc2 \lm(\gamma_t,D\sm D_1;D) 
 \Bigr\}.
\end{equation}
We now explain the terms here.  First, the partition
functions $\Psi_{  D_1 \setminus \gamma_t}(\gamma(t),
    \tilde w)$ and $\Psi_{D \setminus \gamma_t}(\gamma(t),
      w)$ do not exist because $D \setminus \gamma_t$
  is not locally analytic at $\gamma(t)$.  However, the
  ratio of partition functions is well defined 
 using the rule \eqref{confconv},
\begin{equation}  \label{jun3.1}
  \frac{\Psi_{  D_1 \sm \gamma_t}(\gamma(t),
    \tilde w)}{\Psi_{D \sm \gamma_t}(\gamma(t),
      w)}  = 
        \frac{|f'(  w_1)|^{b_{ w_1}} \,
        \Psi_{  f(D_1 \sm \gamma_t)}(f(\gamma(t)),
    f( w_1))}{|f'(w)|^{b_w} \, 
    \Psi_{f(D \sm \gamma_t)}(f(\gamma(t)),
      f(w))} , 
      \end{equation}
      where $f:D \sm \gamma_t \rightarrow
      f(D \sm \gamma_t)$ is a conformal 
      transformation. 
   The parameter   $\cc  
   =(3\kappa-8)b$ is the \emph{central charge} (a
parameter from conformal field theory) and
$\lm(\gamma_t,D\sm D_1;D)$ is a geometric quantity, namely,
the Brownian loop measure (introduced in~\cite{LW}) of
the set of loops in $D$ which intersect both~$\gamma_t$
and~$D\sm D_1$.  For chordal SLE in simply connected domains we could equivalently consider Werner's SLE$_{8/3}$ loop measure~\cite{Werner} for $\Lambda$, which amounts to considering the outer boundaries of the Brownian loops.  For reversed radial SLE as presented in this paper, and for SLE in multiply connected domains~\cite{L11}, one sees topologically that it is no longer equivalent to consider the SLE$_{8/3}$ loop measure for $\Lambda$; the Brownian loop measure must be used instead.

If $D,D_1$ are simply connected, $w_1 = w \in \partial D$,
and $\partial D$ and $\partial D_1$ agree in a neighborhood
of $w$, then we can let $t = \infty$ and see that $\mu_1
\ll \mu$ with 
\begin{equation}
\label{bpert}
\frac{d\mu_1}{d\mu } (\gamma)
= 1\{\gamma\subset D_1\}
\exp \Bigl\{ \frac\cc2 \lm(\gamma,D\sm D_1;D) \Bigr\}.
\end{equation}
This  formula inspired the definition of \SLEk\ in
multiply connected domains that appeared in~\cite{L11}.
Essentially, for general domains $D$ with $z\in D$ and
$w\in D$ or $w$ a smooth boundary point, one defines
$\mu_D(z,w)$ so that~\eqref{bpert} and the conformal
covariance rule~\eqref{confconv}
hold across all domains.
The consistency of this definition is easy to check.
The definition does not immediately establish that
it is a finite measure, but this has been proved for
$\kappa \leq 8/3$ (in which case $\cc \leq 0$) and
for simply and doubly connected domains if $8/3 < \kappa
\leq 4$. 

To prove these results, one considers the case $D = \Half$,
$z=0$, $w= \infty$ with (by normalization) $\Psi_\Half(0,\infty) = 1$.
Then, given an initial segment $\gamma_t$, let $g_t:
\Half \setminus \gamma_t \rightarrow \Half$ be a conformal
transformation with $g_t(z)  = z + o(1)$ as $z \rightarrow \infty$.
Then, $g'(\infty) = 1$ and the ratio in \eqref{jun3.1} becomes
\[   K_t :=  \frac{\Psi_{D_1 \setminus \gamma_t}(\gamma_t, w_1)}
{\Psi_{\Half \setminus \gamma_t}(0,\infty)} =
   |g_t'(w_1)|^{b_{w_1}} \, \Psi_{g_t(D_t)}(U_t,f(w_1)), \]
   where $U_t = g_t(\gamma(t))$.  
 The probability measure $\mu_{D_1}^{\#}(0,w_1)$ is obtained
 by weighting by $K_t$ using the Girsanov theorem.  Since $K_t$
 is not a local martingale, one must first multiply by a 
 compensator, and computing this
gives the Brownian loop term.  That is to say, the term
in \eqref{jun3.2} considered as a function of $t$ is a local
martingale.  When one uses the Girsanov theorem, one finds that
one can define the probability measure  $\mu_{D_1}^{\#}(0,w_1)$
as a solution to the Loewner equation with a driving function
$U_t$ that has a drift.  These new processes are sometimes
called SLE$(\kappa,\rho)$ processes.
The method described in this paragraph was introduced in~\cite{LSWrest} in a slightly different form.

Another reason to consider SLE from the partition function
point of view is to compare it to discrete models at
criticality.
One expects the discrete models that converge to SLE to
have partition functions which converge to the SLE
partition functions when one normalizes by a power of the
lattice spacing.
This agrees with the power-law conformal covariance rule
for SLE partition functions.
Moreover, the Radon--Nikodym derivative~\eqref{jun3.2} of two
SLE measures on an initial segment
is entirely analogous to what occurs in
families of finite measures on paths in a discrete lattice
such as self-avoiding walk, loop-erased walk,
$\lambda$--self-avoiding walk and the percolation
exploration process.
Each of these families has a weight function $W$ such
that its measures $\mu_D(z,w)$ are related by the
first-step decomposition
\[
\mu_D(z,w) = \sum W(D, [z,\zeta]) \,\,
 [z,\zeta] \oplus \mu_{D\sm[z,\zeta]}(\zeta,w),
\]
in which the sum runs over all vertices $\zeta$ that adjoin
the starting point $z$.
In such families we have the Radon--Nikodym derivative
\[
\frac{d\mu_{D_1}(z,w_1)}
     {d\mu_{D  }(z,w  )} (\gamma_t)
=
\frac{W(D_1,\gamma_t)}
     {W(D,  \gamma_t)} \,
\frac{\tmass_{D_1\sm\gamma_t}(\gamma(t),w_1)}
     {\tmass_{D  \sm\gamma_t}(\gamma(t),w  )},
\]
where 
\[
W(D,\gamma_t)
=
\prod_{s=0}^{t-1} W(D\sm\gamma_s, [\gamma(s),\gamma({s+1})]).
\]
In fact, in the discrete models mentioned above, this last
quantity is a functional of the random walk loop measure.

If $z,w$ are boundary points and $D$ is simply connected,
Dapeng Zhan showed in~\cite{ZhanC} that $\mu_D(w,z)$ can be
obtained from $\mu_D(z,w)$ by reversing the paths.  
The argument also shows that in the probability measure
$\mu_D(z,w)$, the conditional distribution given both an initial
segment  and a terminal segment is chordal SLE in the
slit domain connecting the interior endpoints of the paths.
If $z \in D$ and $w \in \partial D$, it was suggested
in \cite{LJSP} to define $\mu_D(w,z)$ to be the reversal
of radial $\mu_D(z,w)$.   This definition was validated
by Zhan~\cite{ZhanWP}, who constructed a probability
measure on curves connecting boundary points on an annulus.
These measures satisfy the condition that, given an initial
and a terminal segment, radial SLE is distributed like
annulus SLE in the remaining domain.  In \cite{L11} it was
shown that this measure is the same as the probability
measure defined in \cite{LJSP}.  In particular, it was
shown that the annulus partition function is finite.

In this paper, we take a different, 
but as we show equivalent, approach to defining
reversed radial SLE by giving its Radon-Nikodym
derivative with respect to whole-plane SLE.  Here we are
using whole-plane SLE as the natural base measure for paths
starting at an interior point in the same way that
chordal SLE in $\Half$ is the base measure for SLE starting
at a boundary point.  
Motivated by formulas such as~\eqref{jun3.2}, we would
like to be able to write the Radon-Nikodym derivative of
reversed radial SLE with respect to whole-plane SLE as
\begin{equation}
\label{badrnd}
\frac{d\mu_D(0,w)}{d\mu_\C(0,\infty)}(\gamma_T)
= 
\exp \Bigl\{ \frac\cc2 \lm(\gamma_T,\C\sm D;\C) \Bigr\}
\frac{\tmass_{D\sm\gamma_T}(\gamma(T),w)}
     {\tmass_{\C\sm\gamma_T}(\gamma(T),\infty)}.
\end{equation}
This is false as written, because $\lm(\gamma_T,\C\sm
D;\C)$ is infinite.

To make sense of this Radon--Nikodym derivative, we
introduce a finite normalized quantity
$\lmstar(\gamma_T,\C\sm D)$ based on the loop measure,
which has many of the properties we want.  This is
similar  in spirit to ``Wick products''.
\begin{theorem} \label{main}
 If $V_1,V_2$ are disjoint nonpolar
closed subsets of the
Riemann sphere, then the limit
\begin{equation}  \label{feb28.1}
    \Lambda^*(V_1,V_2) = \lim_{r \downarrow 0}
             [\Lambda(V_1,V_2;\Outer_r) - \log \log(1/r)],
\end{equation}
exists where
\[   \Outer_r = \{z \in \C: |z| > r\}. \]
Moreover, if $f$ is a M\"obius transformation of the
Riemann sphere,
\[   \Lambda^*(f(V_1),f(V_2)) = \Lambda^*(V_1,V_2).\]
\end{theorem}
We could write the
assumption ``disjoint nonpolar closed subsets of the Riemann
sphere'' as ``disjoint closed subsets of $\C$, at
least one of which is compact, such that 
Brownian motion hits both
subsets at some positive time''.  

Roughly stated, the Brownian loop measure is infinite both because of short loops and because of long loops.  We remark that the loop measure term $\Lambda(V_1,V_2;\Outer_r)$ is necessarily finite in~\eqref{feb28.1}: intuitively, having $V_1,V_2$ disjoint prevents short loops, and long loops are very likely also to leave $\Outer_r$ at some point and hence not contribute to the loop measure term.

By M\"obius invariance, $\Lambda^*(V_1,V_2)$ could equally well be defined by shrinking down around a point other than the origin, or by replacing $\Lambda(V_1,V_2;\Outer_r)$ with the mass of loops hitting~$V_1$ and~$V_2$ that stay in a disk of large radius~$1/r$ as~$r\downarrow0$.  Indeed, this is how we prove the M\"obius invariance in Sect.~\ref{sec:blm}.

Having introduced $\lmstar$, we can
reformulate~\eqref{badrnd} correctly.
\begin{theorem}
\label{slethm}
Let $\kappa\le4$.  Let~$D$ be a simply connected domain
containing~$0$ and $w\in\p D$ a smooth boundary point.
Let~$T$ be a stopping time 
for whole-plane \SLEk\ from $0$ to $\infty$ such
that~$\gamma$ does not leave~$D$ by time~$T$.  Then the
Radon-Nikodym derivative of reversed radial \SLEk\ with
respect to whole-plane \SLEk\ up to time~$T$ is
\begin{equation}
\label{revradrnd}
\frac{d\mu_D(0,w)}{d\mu_\C(0,\infty)}(\gamma_T)
= c_1 
\exp \Bigl\{ \frac\cc2 \lmstar(\gamma_T,\C\sm D) \Bigr\}
\frac{\tmass_{D\sm\gamma_T}(\gamma(T),w)}
     {\tmass_{\C\sm\gamma_T}(\gamma(T),\infty)},
\end{equation}
where~$c_1$ is a constant depending only on $\kappa$.
\end{theorem}

This theorem is also relevant to boundary/bulk SLE, the
natural generalization of radial SLE to non--simply
connected domains.
\begin{corollary}
\label{boundarybulkcor}
Let $D$ be a complex domain containing $0$, not necessarily
simply connected, and $w\in\p D$ a smooth boundary point.
Then the measure $\mu_D(0,w)$ defined by~\eqref{revradrnd}
is the reversal of boundary/bulk SLE $\mu_D(w,0)$, 
as defined in~\cite{L11}.
\end{corollary}

We now describe the structure of the paper.
In Sect.~\ref{sec:prelim} we introduce some notation and
present a number of preliminary results about Brownian
motion, conformal mapping, and SLE, many of which have been
proved elsewhere.
In Sect.~\ref{sec:sle} we deal with reversed radial SLE and
prove Theorem~\ref{slethm}.
In Sect.~\ref{sec:blm} we deal with the normalized loop measure independently of any SLE notions and prove  Theorem~\ref{main}.
Proposition~\ref{good-lmstar-estimate}, which is an
estimate for the normalized loop measure of loops hitting
both boundary components of a conformal annulus, is proved 
in Sect.~\ref{subs:lmstar-est}.
This estimate is more precise than Theorem~\ref{slethm}'s
proof requires, and may be of independent interest.

\section{Preliminary results}
\label{sec:prelim}

We will use the following notation:
\begin{gather*}
    \Disk_r = \{z: |z| < r\},  \qquad  \Disk = \Disk_1,  
    \qquad \Half = \{z: \operatorname{Im} z > 0\}, \\
    \Outer_r = \{z: |z| > r\}, \qquad  \Outer = \Outer_1,
          \qquad \Outer_r(w) = w + \Outer_r, \\
    A_{r,R} = \Disk_R \cap \Outer_r = \{z: r < |z| < R\},
          \qquad A_R = A_{1,R}, \\
     C_r = \p \Disk_r = \p \Outer_r = \{z:|z| = r\},
          \qquad C_r(w) = \p \Outer_r(w) = \{z: |w-z| = r\}.
\end{gather*}
For $S\subset\C$, we denote the complement $\C\sm S$
by~$S^c$.

The implicit constants in all $O(\cdot)$ terms are
universal unless otherwise stated.
The constants in a $O_r(\cdot)$ term may depend on $r$ but
not on any other quantity.
The notation $x\asymp y$ means that there is a universal
constant $c>0$ such that $c\inv<x/y<c$.

\subsection{Complex Brownian motion}
\label{sec:cbm}

We say that a subset $V$ of $\C$ is {\em nonpolar} if it is
hit by Brownian motion.  More precisely, $V$ is nonpolar if
for every $z\in \C$, the probability that a Brownian motion
starting at $z$ hits $V$ is positive.  Since Brownian
motion is recurrent we can replace ``is positive'' with
``equals one''.  In a slight abuse of terminology, we will call a domain
(connected open subset) $D$ of $\C$ nonpolar if $\p D$ is
nonpolar.

\subsubsection{Harmonic measure and excursion measure}
\label{sec:harmmeas}

If $B_t$ is a complex Brownian motion and $D$ is a domain,
let
\[
\tau_D = \inf\{t: B_t \notin D\}.
\]
A domain $D$ is nonpolar if and only if $\Prob^z\{\tau_D
< \infty\}=1$ for every $z$.  In this case we define
harmonic measure of $D$ at $z \in D$ by
\[
h_D(z,V) = \Prob^z\{B_{\tau_D} \in V\}.
\]
If $V$ is smooth then we can write
\[
h_D(z,V) = \int_V h_D(z,w) \, |dw|,
\]
where $h_D(z,w)$ is the {\em Poisson kernel}.  If $z \in \p
D \setminus V$ and $\p D$ is smooth near~$z$, we define the
{\em excursion measure} of $V$ in $D$ from $z$ by
\[
\exc_D(z,V) = \exc(z,V;D) = \p_{\bf n} h_D(z,V),
\]
where ${\bf n} = {\bf n}_{z,D}$ denotes the unit inward
normal at $z$.  If $V$ is smooth, we can write
\[
\exc_D(z,V) = \int_V h_{\p D}(z,w) \, |dw|,
\]
where $h_{\p D}(z,w) := \p_{\bf n}h_D(z,w)$ is the {\em
excursion} or {\em boundary Poisson kernel}.  (Here the
derivative $\p_{\bf n}$ is applied to the first variable.)
One can also obtain the excursion Poisson kernel as the
normal derivative in both variables of the Green's
function; this establishes symmetry, $h_{\p D}(z,w) = h_{\p
D}(w,z)$.  If~$f: D \rightarrow f(D)$ is a conformal
transformation, then (assuming smoothness of~$f$ at
boundary points at which $f'$ is taken)
\begin{align*}
h_D(z,V)      &= h_{f(D)}(f(z),f(V)),  \\
h_D(z,w)      &= |f'(w)| \, h_{f(D)}(f(z), f(w)), \\
\exc_D(z,V)   &= |f'(z)| \, \exc_{f(D)}(f(z),f(V)), \\
h_{\p D}(z,w) &= |f'(z)| \, |f'(w)| \, h_{\p 
    f(D)} (f(z), f(w)).
\end{align*}

\subsubsection{Brownian bubble measure}
\label{sec:brownbub}

If $D$ is a nonpolar domain and $z \in \p D$ is an analytic
boundary point (i.e., $\p D $ is analytic in a neighborhood
of $z$), the Brownian bubble measure~$m_D(z)$ in~$D$ at~$z$
is a sigma-finite measure on loops
$\gamma:[0,t_\gamma]\rightarrow \C$ with $\gamma(0)
= \gamma(t_\gamma) = z$ and $\gamma(0,t_\gamma)\subset D$.
It can be defined as the limit as $\epsilon \downarrow 0$
of $\pi \, \epsilon^{-1} \, h_D(z + \epsilon {\bf n},z)$
times the probability measure on paths obtained from
starting a Brownian motion at $z + \epsilon {\bf n}$ and
conditioning so that the path leaves~$D$ at~$z$.  Here
${\bf n} = {\bf n}_{z,D}$ is the inward unit normal.  If
$\tilde D \subset D$ agrees with $D$ in a neighborhood of
$z$, then the bubble measure in~$\tilde D$ at~$z$,
$m_{\tilde D}(z)$, is obtained from $m_D(z)$ by
restriction.  This is also an infinite measure but the
difference $m_D(z) - m_{\tilde D}(z)$ is a finite measure.
We will denote its total mass by
\[
m(z;D,\tilde D) = \|m_D(z) - m_{\tilde D}(z)\|.
\]
The normalization of $m$ is chosen so that
\begin{equation}  \label{mar2.13}
  m(0;\Half, \Half \cap \Disk) = 1.
\end{equation}

\begin{remark}
The factor of $\pi$ in the bubble measure was put in so
that \eqref{mar2.13} holds.  However, the loop measure in
the next section does not have this factor, so we will have
to divide it out again.  For this paper, it would have been
easier to have defined the bubble measure without the $\pi$
but we will keep it in order to match definitions
elsewhere.
\end{remark}

The bubble measure is conformally covariant~\cite{LW}: if
$f:D\to f(D)$ is conformal and $z\in\p D$ and $f(z)$ are
smooth boundary points, then
\begin{equation}
\label{bubcc}
f\circ m_D(z) = |f'(z)|^2 \, m_{f(D)}(f(z)).
\end{equation}

\subsubsection{Brownian loop measure}
\label{sec:brownloop}

\begin{definition}
A \emph{rooted loop} in a domain $D\subset\C$ is a
continuous map $\gamma:[0,t_\gamma]\to D$ with $t_\gamma>0$
and $\gamma(0)=\gamma(t_\gamma)$.  Its \emph{root} is
$\gamma(0)$.

An \emph{unrooted loop} in $D$ is an equivalence class of
rooted loops in $D$ under the equivalence
$\gamma\sim\gamma_s$ for all $s$, where
$\gamma_s(t)=\gamma(s+t)$ (considering $\gamma$ as a
$t_\gamma$-periodic function) and $t_{\gamma_s}=t_\gamma$.
\end{definition}

The Brownian loop measure $\loopm_D$ is a sigma-finite
measure on unrooted loops in a domain $D$.  The
measure~$\loopm_\C$ can be defined as follows.
\begin{itemize}
\item  Consider the measure on triples
 $(z,
t_\gamma, \tilde\gamma)$ given by
\[ 
\text{area} \times \frac{dt}{2\pi t^2} 
\times \left( \text{length $1$ Brownian bridge from $0$ in $\C$}\right).
\]
\item Let
\[  \gamma(s) = z+\sqrt{t_\gamma\mathstrut}\,\tilde\gamma (  s
/t_\gamma).\]
\item Project this measure onto unrooted loops by forgetting
the root.
\end{itemize}
Then~$\loopm_D$ is defined to be~$\loopm_\C$ restricted to
loops in~$D$.

We have defined the measure so that it satisfies the  restriction
property: if $D' \subset D$, then $\loopm_{D'}$ is $\loopm_D$
restricted to loops in $D'$. 
The other important feature of the Brownian loop measure is
its  conformal invariance, which was proved in~\cite{LW},
Proposition~6.
\begin{proposition}
If $f:D\to f(D)$ is a conformal map, then
$f\circ\loopm_D=\loopm_{f(D)}$.
\end{proposition}

For computational purposes it is useful to write the
measure in terms of the bubble measure, which can be done
in many ways.  We will use the following expression,
which assigns to each unrooted loop the root furthest from
the origin:
\begin{equation}
\label{furthestroot}
\loopm_\C
= \frac1\pi \int_0^{2\pi}\int_0^\infty m_{\D_r}(re^{i\theta})
 \,r \,dr \,d\theta.
\end{equation}
(For a proof, see~\cite{LW}, Proposition~7, and
apply~\eqref{bubcc}.)
To be precise, we are considering the right hand side as a
measure on unrooted loops.
We can also assign to each unrooted loop the root closest
to the origin:
\begin{equation}  \label{march1.1}
  \loopm_\C = \frac 1 \pi \int_0^{2\pi}
   \int_0^\infty m_{\Outer_r}(re^{i\theta}) \, r \, dr\,
  d \theta .
\end{equation}
If $\overline \Disk_r \subset D$, then the Brownian loop
measure in $D$ restricted to loops that intersect
$\overline \Disk_r$ can be written as
\begin{equation} \label{feb27.3.new}
   \frac 1 \pi \int_0^{2\pi}
   \int_0^r m_{D_s}(se^{i\theta}) \, s \, ds\,
  d \theta , 
\end{equation}
where $D_s = D \cap \Outer_s$.  If $r_1 < r$, then the
Brownian loop measure restricted to loops in $\Outer_{r_1}$
that intersect $\overline \Disk_r$ is given by
\[  \frac 1 \pi \int_0^{2\pi}
   \int_{r_1}^r m_{D_s}(se^{i\theta}) \, s \, ds\,
  d \theta .\] 
Using this and appropriate properties of the bubble
measure we can conclude the following.

\begin{lemma}    \label{diamlem}
  For every $0 < s< r  < \infty$ and $d >0$, the
loop measure of the set of loops in 
$\Outer_s$ of diameter
at least $d$ that intersect $\Disk_r$ is finite. 
\end{lemma}

\begin{remark}  This result is not true for $s = 0$.  The
Brownian loop measure of loops in $\C$ of diameter
greater than $d$ that intersect the unit disk is
infinite.  See, e.g., Lemma \ref{feb25.lem1} below.
\end{remark}

Conformal invariance implies that the Brownian loop
measure of loops in~$A_{r,2r}$
that separate the origin from infinity
is the same for all $r$.
It is easy to see that this measure is positive and the
last lemma shows that it is finite. It follows that
the measure of the set of loops that surround the origin
is infinite.

If $V_1,V_2,\ldots$ are closed subsets of the Riemann sphere
and $D$ is a nonpolar domain, then
\[   \Lambda(V_1,V_2,\ldots,V_k; D) \]
is defined to be the loop measure of the set of loops in
$D$ that intersect all of the sets $V_1,\ldots,V_k$. 
Note that
\begin{multline}  \label{cascade}
 \Lambda(V_1,V_2,\ldots,V_{k}; D)  \\
 = \Lambda(V_1,V_2,\ldots,V_{k+1}; D)
 +  \Lambda(V_1,V_2,\ldots,V_{k}; D \setminus
  V_{k+1}) .
\end{multline}
  If $V_1,V_2,
\ldots, V_k$ are the traces of simple curves
that pass through the origin, then the comment in the last
paragraph shows that for all $r > 0$,
\[    \Lambda(V_1,V_2,\ldots, V_k;\Disk_r) = \infty.\]

\subsection{Conformal mapping}
\label{sec:confmap}

\subsubsection{Univalent functions and capacity}  \label{unisec}

A \emph{univalent function} is a one-to-one holomorphic
function.
We will need the following version of the growth and
distortion theorems for univalent functions.
For a proof, see~\cite{Lconf}, Theorem~3.21 and
Proposition~3.30.

\begin{proposition}
\label{growthdist}
If $f$ is univalent on $\Outer_\rho$,
$f(\infty)=\infty$ and $f'(\infty)=1$, then if
$r=\rho/|z|<1$,
\[
\begin{gathered}
f(z)=z+O(\rho),\\
\biggl(\frac{1-r}{1+r}\biggr)^3
\le |f'(z)| \le \biggl(\frac{1+r}{1-r}\biggr)^3.
\end{gathered}
\]
\end{proposition}

We will use the fact that~$|f(z)/z|$ and~$|f'(z)|$ are
both~$1+O(r)$ as~$r\to0$, uniformly in~$f$.
By the Koebe 1/4--theorem, this is as true for~$f\inv$ as
for~$f$.

\begin{definition}
\label{hulldef}
A \emph{hull} is a compact, connected set $K \subset \C$
larger than a single point.
We denote by $g_K$ the unique conformal map $g_K:\C\sm
K\to\Outer_t$, for some $t>0$, with $g_K(\infty)=\infty$
and $g_K'(\infty)=1$.
The \emph{capacity} of~$K$ is defined by $\ccap K =\log t$.
\end{definition}
\begin{remark}
\label{radius-bound}
If $0\in K$ and $\ccap K=\log t$, then the \emph{radius} of
$K$, $\rad K:=\max\{|k|:k\in K\}$, lies in $[t,4t]$ by the
Schwarz lemma and the Koebe 1/4--theorem.
In particular, $|g_K(z)/z|$ and $|g_K'(z)|$ are $1+O(t/|z|)$ as
$t/|z|\to0$.
\end{remark}

\subsubsection{Conformal annuli}  \label{annulisec}

In this section we let $\delta_t = 1/\log(1/t)$, let
$\domain$ denote the set of simply connected domains $D$
containing the origin with $\dist(0, \p D) = 1$, and let
$\hulls_t$ denote the set of hulls $K \subset \Disk$
 of capacity $\log t$ containing the origin.

If $D \in \domain$, let $\psi = \psi_D:
D \rightarrow \Disk$ denote the unique conformal
transformation with $\psi(0) = 0$ and $\psi'(0) > 0$.
If $D \in{\domain}$ and $K \in \hulls_t$, let $\phi$ denote
a conformal transformation $\phi = \phi_{D,K}:
D \setminus K \rightarrow A_{s,1}$.
It is well known that this $\phi$ is defined uniquely up to
a final rotation, and in particular, $s = s_{D,K}$ is a
uniquely defined number reflecting the conformal type of
the conformal annulus $D\sm K$.
We recall the classical fact that nested conformal annuli
have nested values 
of $s$.

\begin{lemma}
\label{ratiocomp}
In this situation, $s\asymp t\asymp \rad K$.
In particular, the expressions $O(s)$ and $O(t)$ are
interchangeable, as are ``\/$s\to0$'' and ``\/$t\to0$''.
\end{lemma}
\begin{proof}
Let $r=\rad K$.
Remark~\ref{radius-bound} provides the bound $t\asymp r$.

Since $A_{r,1}\subset D\sm K$, $s\le r$.
Applying $\psi$ and using the Koebe $1/4$-theorem, it
suffices to prove that $t=O(s)$ in the case $D=\D$.
Now $g_K(\D\sm K)\subset A_{t,1+O(t)}\subset A_{t,O(1)}$ by
Proposition~\ref{growthdist}, and thus $t=O(s)$.
\maybeqed\end{proof}

\begin{lemma}
\label{exc-est}
Let $K\in\hulls_t$, $s=s_{\D,K}$, $0<r\le|z|<1$, $|w|=1$ and $t\to0$.
Then
\[
\begin{aligned}
s &= t\,[1+O(t)],
\\
h_{\D\sm K}(z,K) &= \delta_t\,\log|1/z|\,[1+O_r(t)],
\\
\exc_{\D\sm K}(w,K)
&= \delta_t  \,[1+ O(t)].
\end{aligned}
\]
\end{lemma}

\begin{proof}
Since
\[
h_{A_{r,1}}(z,C_r)=\frac{\log|1/z|}{\log(1/r)},
\]
it suffices to consider $|z|=r$.
By conformal invariance of Brownian motion,
\[
h_{\D\sm K}(z,K)=h_{g_K(\D\sm K)}(g_K(z),C_t).
\]
By Proposition~\ref{growthdist}, $|g_K(y)|=1+O(t)$ for
$y\in C_1$, and $|g_K(z)|=r\,[1+O_r(t)]$.
It follows that
\[
A_{t,1-O(t)}\subset g_K(\D\sm K)\subset A_{t,1+O(t)}.
\]
We conclude that $s = t \, [1+O(t)]$ and
\[
h_{\D\sm K}(z,K)
=\frac{\log(1/r\,[1+O_r(t)])}{\log(1/t\,[1+O(t)])}
=\delta_t\,\log(1/r)\,[1+O_r(t)].
\]
The remaining estimates follow.
\maybeqed\end{proof}

We define an \emph{inner transformation} of a conformal
annulus $A\subset\C$ to be a conformal map of $A$ fixing
and preserving the orientation of one of its boundary
components, which we call the \emph{outer boundary}.

\begin{proposition}
\label{outerboundary}
Let $A$ be a conformal annulus of conformal type $s$, 
and $s\to0$.
Then for any inner transformation $\chi$ of $A$,
$|\chi'|=1+O(s)$ on $A$'s outer boundary.
\end{proposition}

\begin{proof}
By a conformal map, we may assume that $A\subset\D$ with
outer boundary $\p\D$.
Moreover, it suffices to consider $\chi:A\to A_{s,1}$ and
to estimate the derivative at $1$.
Let $K:=\D\sm A$ have capacity $\log t$, and let $L(z)
= \log \chi(e^{iz}) = u(z) + i v(z)$, defined in a
neighborhood of the origin using Schwarz reflection.
Since $u$ vanishes on the real axis, $\p_xu(0) = 0$.
Since
\[
\frac{\log |\chi(z)|} 
   {\log s} = h_{A_{s,1}}(\chi(z),C_s) 
   = h_{\D \setminus K}(z,K) ,
\]
we see that
\[  
-\p_yu(0) = \log (1/s) \, \exc_{\D \setminus K}(1,K)
 = \frac{\delta_ t}{\delta_ s} \, [1 + O(t)] = 1 + O(t) .
\]
Therefore, $L'(0) = i + O(t)$.
By the chain rule, $|\chi'(1)| = 1 + O(t).$
\maybeqed\end{proof}

The following estimate allows us to extend
Lemma~\ref{exc-est} to general domains $D\in\domain$.

\begin{lemma}
\label{NL}
Let $K\in\hulls_t$.
Let $f$ be univalent on $\D$.
Let $t\to0$.
Then
\[
\ccap f(K) = \log(|f'(0)|\,t) + O(t).
\]
\end{lemma}

\begin{proof}
By a similarity we may assume that $f(0)=0$ and $f'(0)=1$.
By the growth theorem, $f(z)=z\,[1+O(z)]$ as $z\to0$, and hence
$\rad f(K) = [1+O(t)]\, \rad K$.
By Remark~\ref{radius-bound}, $\ccap f(K)
= \ccap K + O(1)$.

Let $B_s$ be a complex Brownian motion.
When we ask to start $B_s$ at infinity, we mean to let the
starting point be distributed uniformly on a circle of
sufficiently large radius.

Let $k$ be a constant to be determined.  Let
\begin{align*}
\sigma&=\inf\{s:B_s\in K\}, &
\tilde\sigma&=\inf\{s:B_s\in f(K)\}, \\
X &= \log |kt/f(B_\sigma)|, &
\tilde X &= \log |kt/B_{\tilde\sigma}|, \\
X_0 &= \E^\infty[X], &
\tilde X_0 &= \E^\infty[\tilde X].
\end{align*}
We will use the fact that $\ccap K = \E^\infty[\log
|B_\sigma|]$ and $\ccap f(K) = \E^\infty[\log
|B_{\tilde\sigma}|]$ (for a proof, see~\cite{Lconf},
Corollary~3.34).
Since $\log|B_\sigma|=\log|f(B_\sigma)|+O(t)$, it suffices
to prove that $X_0 = \tilde X_0 + O(t)$.
We already know that $X_0=\tilde X_0+O(1)$.

Let
\begin{align*}
\tau&=\inf\{s:B_s\in C_{kt}\}, &
\tilde\tau&=\inf\{s:B_s\in f(C_{kt})\}, \\
\rho&=\inf\{s>\tau:B_s\in C_{1/2}\}, &
\tilde\rho&=\inf\{s>\tilde\tau:B_s\in f(C_{1/2})\}, \\
V&=\{\rho<\sigma\}, &
\tilde V &= \{\tilde\rho<\tilde\sigma\}.
\end{align*}
By Proposition~\ref{growthdist} applied to $g_K$, we may
choose the universal constant $k>8$ so that
$\Prob^z(V) \asymp \delta_t$ for $z\in C_{kt}$, and we may
also insist that $f(K)\subseteq C_{kt/2}$.
This implies that $\tilde X_0 \asymp 1$ and hence
\begin{equation}
\label{theta1}
\E^z[\tilde X] \asymp 1, \qquad 3kt/4 < |z| < 5kt/4,
\end{equation}
by Harnack's inequality.
Let
\[
u(z) = \E^z[X]-\tilde X_0, \qquad
u^* = \max_{z\in C_{1/2}} |u(z)|,
\]
which exists since $u$ is harmonic off $K$.

Let $z\in C_{1/2}$.
Denote by $\mu$ the distribution of $f(B_\tau)$ if $B_s$ is
started at~$z$.  
By the strong Markov property and conformal invariance of
Brownian motion, $\Prob^z(V)=\Prob^\mu(\tilde V)$ and also
\begin{equation}
\label{maineq}
\E^z[X] = \E^\mu[\tilde X] 
+ \Prob^z(V) \bigl(\E^z[X\mid V]-\E^\mu[\tilde X\mid \tilde V]\bigr).
\end{equation}
We claim
\begin{align}
\E^\mu[\tilde X] &= \tilde X_0 + O(t), 
\label{est1} \\
\E^\mu[\tilde X\mid \tilde V] &= \tilde X_0 + O(t), 
\label{est2} \\
\bigl|\E^z[X\mid V] - \tilde X_0\bigr| &\le u^*.
\label{est3}
\end{align}
The strong Markov property applied to $\rho$
yields~\eqref{est3}.
The strong Markov property applied to $\tilde\rho$, the
estimate 
\[
h_{\Outer_{kt}}(w,kte^{i\theta}) = [1 +
O(t)]/2\pi kt,
\qquad
|w| \ge 1/8,
\]
and~\eqref{theta1} together yield~\eqref{est2}.
If $w=O(t)$, we know $f(w)=w+O(t^2)$.
Moreover, if $h(w)=\E^w[\tilde X]$ then $|\nabla h(w)| =
O(1/t)$ for $7kt/8 < |w| < 9kt/8$ by~\eqref{theta1} and
standard derivative estimates for harmonic functions.
We conclude that $\E^{f(w)}[\tilde X] = \E^w[\tilde X] +
O(t)$ for $w\in C_{kt}$.
It follows that $\E^\mu[\tilde X] = \E^\nu[\tilde X] +
O(t)$ where $\nu$ is the hitting distribution on $C_{kt}$
of Brownian motion started at $z$.
But $\nu$ is uniform up to a relative error of $O(t)$,
which implies that $\E^\nu[\tilde X] = \tilde X_0 + O(t)$.
Combining these observations yields~\eqref{est1}.

If we apply~(\ref{est1}--\ref{est3}) to~\eqref{maineq},
we conclude that
\begin{align*}
\bigl|\E^z[X]-\tilde X_0\bigr|
&\le O(t) + \Prob^z(V)[u^* + O(t)],\\
\text{so}\qquad\qquad\qquad
u^* &\le O(t) + O(\delta_t) \, u^*,
\end{align*}
so $u^* = O(t)$.
Therefore, $X_0 = \tilde X_0 + O(t).$
\maybeqed\end{proof}

\begin{lemma}
\label{L2.8Str}
Let $D\in\domain$, $K\in\hulls_t$, $\psi=\psi_D$ and $s=s_{D,K}$.
Let $0<r<1$ and $z\in D$ with $r\le|z|$.
Let $t\to 0$.
Then
\begin{align*}
s&=\psi'(0)\,t\,[1+O(t)], \\
h_{D\sm K}(z,K) &= \delta_{\psi'(0)\,t} \,\log|1/\psi(z)| \,[1+O_r(t)].
\end{align*}
Moreover, if $w\in\p D$ is a smooth boundary point,
\[
\exc_{D\sm K}(w,K) = \delta_{\psi'(0)\,t} \,|\psi'(w)| \,[1+O(t)].
\]
In particular, if also $\tilde K\in \hulls_t$ then
$\exc_{D\sm\tilde K}(w,\tilde K)=\exc_{D\sm
K}(w,K)\,[1+O(t)]$.
\end{lemma}

\begin{proof}
Let $\log\tilde t=\ccap\psi(K)$.
By Lemma~\ref{NL}, $\ccap \psi(K) = \log
(\psi'(0)\,t)+O(t)$, so $\tilde t = \psi'(0)\,t\,[1+O(t)]$
and $\delta_{\tilde t} = \delta_{\psi'(0)\,t}\,[1+O(t)]$.
By Lemma~\ref{exc-est}, we conclude that $s=\tilde
t\,[1+O(\tilde t)]$ and
\begin{align*}
h_{D\sm K}(z,K)
&= h_{\D\sm \psi(K)}(\psi(z),\psi(K))
= \delta_{\tilde t}\,\log|1/\psi(z)|\,[1+O_r(\tilde t)],
\end{align*}
using the fact that $|\psi(z)|\ge r/16$.
The remaining estimates follow.
\maybeqed\end{proof}

\subsection{Schramm--Loewner evolution}
\label{sec:introsle}

\subsubsection{Chordal, radial and whole-plane SLE}
\label{sec:chordaletc}

Fix $\kappa>0$.  The \emph{Loewner equation} in the upper
half-plane $\Half$ is the ODE
\[
\p_t g_t(z)=\frac{a}{g_t(z)-U_t},
\quad g_0(z)=z,
\quad z\in\Half,
\]
where $a=2/\kappa$ and the \emph{driving function} $U_t$ is
continuous and real-valued.  Chordal \SLEk\ in~$\Half$ from
$0$ to $\infty$, which we denote by~$\mu_\Half(0,\infty)$, is
the random family of conformal maps $g_t$ induced by this
ODE when~$U_t$ is a standard Brownian motion.

We recall several facts about SLE that were first proved in the seminal paper of Rohde and Schramm~\cite{RS}.  
Chordal \SLEk\ is generated by a continuous curve
$\gamma:[0,\infty)\to\Half$ in the sense that for each $t$,
the set $H_t$ of points $z\in\Half$ for which the solution
exists beyond time~$t$ is the unbounded component of
$\Half\sm\gamma_t$.  (The most delicate case of $\kappa=8$ was 
unresolved in~\cite{RS} but was proved in~\cite{LSWust}.)
Moreover, $\gamma(t)\to\infty$ as
$t\to\infty$.  For $\kappa\le4$ the curve $\gamma$ is
simple, for $4<\kappa<8$ it touches itself, and for
$\kappa\ge8$ it fills the half-plane.

The $g_t$ are easily seen to satisfy the scaling rule
$r\inv\,g_{r^2t}(rz) \eqd g_t(z)$ for $r>0$, and hence
$\gamma\eqd r\gamma$, up to a reparametrization.
As a consequence, the chordal \SLEk\ curve~$\gamma$ from
$z$ to $w$ in any simply connected domain $D$, where
$z,w\in\p D$, may be defined as the conformal image
of \SLEk\ in~$\Half$ from~$0$ to~$\infty$, modulo
reparametrization.
Chordal \SLEk\ from~$z$ to~$w$ in~$D$ is then conformally
invariant and has the domain Markov property:
if $\tau$ is a stopping time, then conditional on the
initial segment $\gamma_\tau$, the remainder of the curve
has the law of \SLEk\ in $D\sm\gamma_\tau$
from~$\gamma(\tau)$ to~$w$.
(In the case $4<\kappa<8$, this means the component of
$D\sm\gamma_\tau$ adjoining~$w$.)

Radial \SLEk\ is a probability measure on curves joining a
boundary point to an interior point of a simply connected
domain~$D$.  It can be obtained from a similar ODE, the
radial Loewner equation in the disk, or can be obtained
from chordal \SLEk\ by weighting by an appropriate local
martingale---see Proposition~\ref{scpfprop}.  In particular,
the paths of chordal and radial SLE from the same point are
locally mutually absolutely continuous.  Radial \SLEk\ is
also conformally invariant and satisfies the domain Markov
property.

Whole-plane \SLEk\ is a variant on radial \SLEk\ in which
the curve $\gamma_t$ connects two marked points on the
Riemann sphere $\Chat$.  The domain Markov property then
takes the following form: conditional on an initial segment
$\gamma_\tau$, the remainder of the curve has the law of
radial \SLEk\ in $\Chat\sm\gamma_\tau$ from $\gamma(\tau)$
to the target point.  This property determines the measure
on curves, up to reparametrization.  We denote
whole-plane \SLEk\ from~$0$ to~$\infty$
by~$\mu_\C(0,\infty)$.  In this paper we adopt the convention that
the whole-plane SLE curve~$\gamma$ is always parametrized so that for each
$t>0$, $\ccap\gamma_t=\log t$.  We will use the notation
$g_t=g_{\gamma_t}$.  As mentioned before, we have $\rad \gamma_t
 \in [t,4t]$ and 
 \[    g_t'(z) = 1 + O(t) , \;\;\;\; t \leq 1/8, \;\;\;|z| \geq 1.\]

\subsubsection{The SLE partition function}
\label{sec:slepartfn}

Let~$D$ be a simply connected domain, 
 $z$ a smooth boundary point and
$w$ either an interior point or a smooth boundary point.
We will consider SLE from~$z$ to~$w$ in~$D$
as a finite measure
\[
\mu_D(z,w)=\tmass_D(z,w) \, \mu_D^\#(z,w),
\]
where $\tmass_D(z,w)$ is the partition function
and $\mu_D^\#(z,w)$ is the corresponding probability
measure.  For chordal and radial \SLEk\ the partition
function is determined by the conformal covariance
rule~\eqref{confconv},
together with the normalizations
$\tmass_\Half(0,\infty)=\tmass_\D(1,0)=1$.

We will frequently use ratios of partition functions
$\tmass_{D'}(z,w')/\tmass_D(z,w)$ even in the case where $\p
D'$ and $\p D$ are not smooth at~$z$.  This ratio is
well-defined so long as $D'$ and $D$ agree in a
neighborhood of~$z$.  Indeed, we can write
\[
\frac{\tmass_{D'}(z,w')}
     {\tmass_{D }(z,w )}
=
\frac{|g'(w')|^{b_{w'}} \, 
      \tmass_{g(D')}\bigl(g(z),g(w')\bigr)}
     {|g'(w )|^{b_{w }} \,
      \tmass_{g(D)}\bigl(g(z),g(w)\bigr)}
\]
for any conformal map $g$ defined on $D\cup D'$ with
$g(z)$ a smooth boundary point, where $b_w,b_{w'}$ are the
appropriate scaling exponents for $w,w'$ in $D,D'$
respectively.  Moreover, in the case of non-smooth boundary points
we will use the notation
\[
\frac{d\mu_{D'}(z,w')}
     {d\mu_{D }(z,w )}
:=
\frac{\tmass_{D'}(z,w')}
     {\tmass_{D }(z,w )}\,
\frac{d\mu_{D'}^\#(z,w')}
     {d\mu_{D }^\#(z,w )},
\]
even though there are no measures 
$\mu_{D }(z,w ),\mu_{D'}(z,w')$ as such.

When we say ``weighting paths locally'' by the
quantity~$Q_t$, we mean the following.  Take the unique
continuous local martingale $M_t=C_t\,Q_t$ with respect to
the implied filtration $\F_t$, where $C_0=1$ and $C_t$ is a
process of bounded variation called the \emph{compensator}
of $Q_t$.  Then consider the paths up to an appropriate
stopping time $\tau$ in the new measure $\Q_\tau$ defined
by $\Q_\tau(A)=\E(AM_{t\wedge\tau})$ for $A\in\F_t$.

Girsanov's theorem says that this change of measure can be
interpreted as adding a drift to the underlying Brownian
motion.
Using this argument, the following proposition can be
proved.
See~\cite{Lparkcity,LJSP,L11} for more detail on the partition function viewpoint of SLE.  
\begin{proposition}
\label{scpfprop}
Let $D$ and $D'$ be simply connected domains agreeing in a
neighborhood $U$ of the smooth boundary point $z$, and let
$w,w'$ be interior or smooth boundary points of $D,D'$
respectively.  Let $\tau$ be a stopping time for the \SLEk\
process $\mu_D(z,w)$ such that $\gamma_\tau$ does not hit
the boundary of $D$ or $D'$ outside $U$.  Then
$\mu_{D'}(z,w')$, up to time $\tau$, can be
obtained by weighting the paths of $\mu_D(z,w)$ locally by
the ratio of partition functions
\begin{equation}
\label{ratiopfgen}
\frac{\tmass_{D'\sm\gamma_t}(\gamma(t),w')}
     {\tmass_{D \sm\gamma_t}(\gamma(t),w )}.
\end{equation}
If also $D=D'$, then the compensator is trivial, so
that~\eqref{ratiopfgen} is the Radon-Nikodym derivative of
the measures on the initial segment $\gamma_t$.
\end{proposition}
\begin{remark}
If $4<\kappa<8$, we ignore any extra connected components
that appear in $D\sm\gamma_t$ for the purpose of the
partition function.
\end{remark}

\subsubsection{Annuli and multiply connected domains}
\label{sec:sleannuli}

From now on let $\kappa\le4$.
Computing the compensator to
\[
\frac{\tmass_{D_1\sm\gamma_t} (\gamma(t),w)}
     {\tmass_{D\sm\gamma_t} (\gamma(t),w)},
\]
 one can
deduce the boundary perturbation rule~\eqref{bpert}.
This rule permits the definition of \SLEk\ in general
domains $D\subset\C$ as in~\cite{L11}.
There is a unique extension of the simply-connected domain
definition so that~\eqref{bpert} and the conformal
covariance rule hold across all domains $D$ and endpoints
$z,w$ (smooth points, if on the boundary).
To be specific, there are two cases defined:
boundary/boundary SLE, generalizing chordal SLE, and
boundary/bulk SLE, generalizing radial SLE.

The case of conformal annuli $D$ and smooth boundary points
$z,w$ on different boundary components is
called \emph{crossing annulus SLE}.  In this case, the
partition function $\tmass_D(z,w)$ is known 
to be finite---see~\cite{L11}---which 
permits us to discuss the crossing annulus SLE probability
measure $\mu^\#_D(z,w)$.  This probability measure is the
same as that used by Zhan in his proof of reversibility of
whole-plane SLE~\cite{ZhanWP}.

\begin{proposition}
\label{pfprop}
Let $D$ be a simply connected domain and $A\subset D$ a
conformal annulus such that $\p D$ is one component of $\p
A$.
Let $z\in\p D$, $w'\in\p A\sm\p D$ and $w\in D$, and
suppose $w'$ is a smooth boundary point.
Let $\tau$ be a stopping time for radial SLE $\mu_D(z,w)$
such that $\gamma_\tau$ does not leave $A$.
On the initial segment~$\gamma_\tau$, the Radon--Nikodym
derivative of crossing annulus SLE $\mu_A(z,w')$ with
respect to radial SLE $\mu_D(z,w)$ is
\[
\frac{d\mu_A(z,w')}
     {d\mu_D(z,w )}(\gamma_\tau)
=
\exp \Bigl\{ \frac\cc2 \lm(\gamma_\tau, D\sm A; D) \Bigr\}
\frac{\tmass_{A\sm\gamma_\tau}(\gamma(\tau),w')}
     {\tmass_{D\sm\gamma_\tau}(\gamma(\tau),w )}.
\]
\end{proposition}

\begin{proof}
Fix a partial curve $\gamma_\tau$ starting from $z$ in~$A$.
Find a simply connected domain $U\subset A$ which contains~$\gamma_\tau$ and agrees
with~$D$ near~$z$ and~$w'$. 
Let $\tilde w\in U$ be arbitrary.
On the partial curves $\gamma_\tau$, we can write down the Radon-Nikodym derivatives 
\[
\frac{d\mu_A(z,w')}{d\mu_U(z,w')},\quad
\frac{d\mu_U(z,w')}{d\mu_U(z,\tilde w)},\quad
\frac{d\mu_U(z,\tilde w)}{d\mu_D(z,\tilde w)}\quad\text{and}\quad
\frac{d\mu_D(z,\tilde w)}{d\mu_D(z,w)}
\]
using Proposition~\ref{scpfprop} and the 
boundary perturbation rule~\eqref{bpert}.
Taking their product 
and using the decomposition rule~\eqref{cascade}
yields the proposition.
\maybeqed\end{proof}

We will also use Theorem~4.6 from~\cite{L11}, which states
that crossing annulus SLE converges to radial SLE:
\begin{theorem}  
\label{anntorad}
There exist $c< \infty$, $q > 0$ such that the following
holds.
Let $t > 0$ and let $\gamma_t$ denote an initial segment of
a path in $\Disk$ starting at $1$ such that if
$g: \Disk \setminus \gamma_t \rightarrow \Disk$ is a
conformal transformation with $g(0) = 0, g'(0) > 0$, then
$g'(0) = e^{t}.$ Suppose that $\log(1/r) \geq t + 2$,
$0 \leq \theta < 2\pi$, and let $\mu_1 = \mu_\Disk(1,0)$,
$\mu_2 = \mu_{A_{r,1}}^\#(1, re^{i\theta})$, both
considered as probability measures on initial
segments~$\gamma_t$.
Let $Y = d\mu_2/d\mu_1$.
Then
\begin{equation}  \label{nov28.1}
  |Y(\gamma_t) - 1| \leq  c\, (re^t)^q. 
  \end{equation}
  Moreover, there exists $c_0 \in (0,\infty)$ such that
\begin{equation}  \label{nov28.2}
  \tmass_{A_{r,1}}(1,re^{ i\theta}) = c_0 \, r^{\btilde - b}
   \, [\log(1/r)]^{ \cc/2}\, [1 + O(r^q)]. 
   \end{equation}
\end{theorem}

\begin{corollary}
\label{convcor}
Radial SLE $\mu_\D(1,0)$ is
the weak limit of 
\[
c_0\inv\, r^{b-\btilde}\,
[\log(1/r)]^{-\cc/2}\, \mu_{A_{r,1}}(1,re^{i\theta})
\]
as $r \downarrow 0$, uniformly in~$\theta$.
\end{corollary}
This follows from Theorem~\ref{anntorad} together with the
continuity of radial SLE at its terminal point, which was
proved in~\cite{Lradcont}.  By Lemma~\ref{ratiocomp}, we conclude:
\begin{corollary}
\label{cor:gamma-convergence}
Let $\gamma$ be a curve starting at $0$ parametrized
so that $\ccap\gamma_t=\log t$, $D$ a simply connected
domain containing~$0$, and $w\in\p D$.
The radial SLE probability measure $\mu^\#_D(w,0)$ 
is the weak limit of the annulus SLE probability measures 
$\mu^\#_{D\sm\gamma_t}(w,\gamma(t))$ as $t \downarrow 0$, uniformly in~$\gamma$.
\end{corollary}

\section{Radial SLE from the interior}
\label{sec:sle}

In this section we prove Theorem~\ref{slethm}, leaving the
proof of Theorem~\ref{main} until Sect.~\ref{sec:blm}.
Throughout this section we assume that $\kappa \leq 4$.
Our approach is to \emph{define} a measure $\mu_D(0,w)$ for
simply connected $D$ that satisfies~\eqref{revradrnd} and
then to prove that it is the reversal of $\mu_D(w,0)$.

\subsection{Definition}
\label{sec:radfromintdefn} 
 
Let $D$ be a simply connected domain containing the
origin.
Recall the normalized loop measure
\[
    \Lambda^*(V_1,V_2) = \lim_{r \downarrow 0}
             [\Lambda(V_1,V_2;\Outer_r) - \log \log(1/r)],
\]
where $\Lambda(V_1,V_2;\Outer_r)$ is the Brownian loop
measure of the loops that hit $V_1$ and $V_2$ but do not
come within distance $r$ of the origin.
\begin{definition}
\label{radfromint}
Let $T$ be a positive stopping time for whole-plane SLE $\mu_\C(0,\infty)$
such that with probability~1, $\gamma_T\subset D$. 
\emph{Radial SLE from the interior of $D$ to $w\in\partial D$} 
is a measure $\mu_D^T(0,w)$ on curves $\gamma_T$ up to time~$T$.  
It is defined by its density with respect to whole-plane SLE, 
which is
\begin{equation}
\label{definition}
\frac{d\mu_D^T(0,w)}
     {d\mu_\C(0,\infty)}(\gamma_T)
= c_1 \exp\Bigl\{ \frac\cc2 \Lambda^*(\gamma_T,D^c) \Bigr\}
\frac{\tmass_{ D\sm\gamma_T}(\gamma(T),w)}
     {\tmass_{\C\sm\gamma_T}(\gamma(T),\infty)}.
\end{equation}
In this formula, the last factor is the ratio of partition
functions for annulus and whole-plane SLE, and $1/c_1=c_0$
is the constant from Theorem~\ref{anntorad}.
\end{definition}

\begin{proposition}
Let $T\ge\tau$ be stopping times for whole-plane SLE such
that with probability~1, $\gamma_T\subset D$.
The measure $\mu_D^T(0,w)$ on $\gamma_T$, considered as a
measure on the initial segment~$\gamma_\tau$, is the same
as the measure $\mu_D^\tau(0,w)$.
\end{proposition}

\begin{proof}
Using the loop measure decomposition~\eqref{loopdecomp}, we
can factor the density as
$\bigl(d\mu_D^T(0,w)/d\mu_\C(0,\infty)\bigr)\,(\gamma_T)=XY,$
where
\[
X=c_1 \exp\Bigl\{ \frac\cc2 \lmstar(\gamma_\tau,D^c) \Bigr\},
\quad
Y=\exp\Bigl\{ \frac\cc2 \lm(\gamma_T,D^c;\gamma_\tau^c) \Bigr\}
\frac{\tmass_{D\sm\gamma_T}(\gamma(T),w)}
     {\tmass_{\C\sm\gamma_T}(\gamma(T),\infty)}.
\]
Note that $X$ is $\F_\tau$--measurable.  By
Proposition~\ref{pfprop}, conditional on $\F_\tau$,
\[
Y=
\frac{d\mu_{D\sm\gamma_\tau}(\gamma(\tau),w)}
     {d\mu_{\C\sm\gamma_\tau}(\gamma(\tau),\infty)}
(\gamma_T),
\quad\text{so}\quad
\E(Y\mid\F_\tau)=
\frac{\tmass_{D\sm\gamma_\tau}(\gamma(\tau),w)}
     {\tmass_{\C\sm\gamma_\tau}(\gamma(\tau),\infty)}
\]
by the domain Markov property for whole-plane SLE. It
follows that
\[
\E\biggl[
\frac{d\mu_D^T(0,w)}
     {d\mu_\C(0,\infty)}(\gamma_T)
\biggm| \F_\tau \biggr]
= X\, \E(Y\mid\F_\tau) =
\frac{d\mu_D^\tau(0,w)}
     {d\mu_\C(0,\infty)}(\gamma_\tau).\qedhere
\]
\end{proof}

This proposition tells us that radial SLE from the interior
is defined consistently across stopping times $T$.
In particular, the total mass of the measure is independent
of $T$.
This means we may consider radial SLE from the interior as
a well-defined measure on curves from~$0$ in~$D$ stopped
before hitting~$\p D$.
We will denote this measure by~$\mu_D(0,w)$ and its
partition function by $\tmass_D(0,w)=\|\mu_D(0,w)\|$.

\subsection{Density estimate}
\label{sec:densityest}

The key to our analysis of radial SLE from the interior is
the following estimate on the normalized loop measure of
the loops that hit both boundary components of a conformal
annulus.
We prove a substantially stronger estimate than is needed
for this paper, because we think that the result is
interesting in its own right.

\begin{proposition}
\label{good-lmstar-estimate}
There exists $c < \infty$ such that if $D \in \domain$,
$t \leq 1/8$ and $K$ is a hull of capacity $\log t$ 
containing the origin, then
\[    \left|\Lambda^*(K, \p D) + \log \log(1/\psi'(0)\,t)  \right| \leq c \,t, \]
where $\psi:D \rightarrow \Disk$ is the unique conformal
transformation with $\psi(0) = 0$ and $\psi'(0) > 0$.
\end{proposition}

The proof is independent of any SLE notions, and we defer
it to Sect.~\ref{subs:lmstar-est}.

We can now estimate the density of radial SLE from the
interior with respect to whole-plane SLE.  Recall the notation 
$\delta_t = 1/\log(1/t)$.

\begin{proposition}
\label{prop:density-estimate}
Let $D \in \domain$ and let $w$
 be a smooth boundary point of $  D$.
Let $\psi$ be a conformal transformation
 from $D$ onto $\D$ fixing $0$. 
Let $\tau$ be a stopping time for whole-plane SLE such that 
with probability $1$, $0<\tau\le t$.
Then, as $t\to 0$,
\begin{equation}
\label{eq:density-estimate}
\frac{d\mu_D(0,w)}{d\mu_\C(0,\infty)}(\gamma_\tau)
= |\psi'(0)|^\btilde \, |\psi'(w)|^b \, \left[1+O(\delta_t)
 \right].
\end{equation}
\end{proposition}

\begin{proof}
Let $t<1/4$, so that whole-plane SLE from~$0$
to~$\infty$ does not leave~$D$ by time~$t$.
By definition
\[ \frac{d\mu_D(0,w)}{d\mu_\C(0,\infty)} (\gamma_t)
  = c_1 \,  \exp\Bigl\{\frac\cc2\Lambda^*(\gamma_t,D^c)\Bigr\}
  \, \frac{\tmass_{D \setminus \gamma_t}(\gamma(t),w)}
     {\tmass_{\C \setminus \gamma_t}(\gamma(t),\infty)}. \]
It follows from Proposition~\ref{good-lmstar-estimate} that
\[
\exp\Bigl\{\frac\cc2\Lambda^*(\gamma_t,D^c)\Bigr\}
=
[\log(1/t)]^{-\cc/2} \,\left[1+ O(\delta_t)\right].
\]

Recall the conformal map $g_t:\C\sm\gamma_t\to\Outer_t$, which
satisfies $g_t(\infty)=\infty$ and $g_t'(\infty)=1$.
If~$U_t:=g_t(\gamma(t))$, we have
\[ 
 \frac{\tmass_{D \setminus \gamma_t}(\gamma(t),w)}
     {\tmass_{\C \setminus \gamma_t}(\gamma(t),\infty)}
 =
\frac{|g_t'(w)|^b \, \tmass_{g_t(D)}(U_t,g_t(w))}
     {\tmass_{\Outer_t}(U_t,\infty)}.\]
Since $|g_t'(w)| = 1 +O(t)$ by Proposition~\ref{growthdist}
and ${\tmass_{\Outer_t}(U_t,\infty)} = t^{\tilde b - b}$,
we get
 \[ \frac{d\mu_D(0,w)}{d\mu_\C(0,\infty)} (\gamma_t)
 = c_1 \, [\log(1/t)]^{-\cc/2}  \, t^{b-\tilde b}\,
  \tmass_{g_t(D)}(U_t,g_t(w))\,
 [1+O(\delta_ t)] .\]
 
Let $\rho:g_t(D) \rightarrow A_{s,1}$ be a conformal
transformation taking $g_t(\p D)$ to~$C_1$.
Then $\rho \circ g_t$ is a conformal transformation of
$D \setminus \gamma_t$ onto $A_{s,1}$.
By Lemma~\ref{L2.8Str} we know that
\[
s = |\psi'(0)|\, t \, [1 + O(t)].
\]
By Proposition~\ref{outerboundary} applied to $\rho\circ
g_t\circ \psi\inv$ and $z\mapsto(t/s)\,\rho(z)$
respectively,
\begin{align*}
|\rho'(g_t(w))| &= |\psi'(w)| \, [1 + O(t)],\\
|\rho'(U_t)| &= |\psi'(0)| \, [1+O(t)].
\end{align*}
Hence, 
\[
\tmass_{g_t(D)}(U_t,g_t(w)) = [1 + O(t)]\,
|\psi'(0)|^b  \,  |\psi'(w)|^b \,
\tmass_{A_{s,1}}(\rho(U_t),\rho \circ g_t(w)). 
\]
Using \eqref{nov28.2} we see that
\[
\tmass_{A_{s,1}}(\rho(U_t),\rho \circ g_t(w)) 
= [1+ O(\delta_t)]\,
c_0 \, |\psi'(0)|^{\btilde - b} \, t^{\btilde - b} \, [\log (1/t)]^{
\cc/2}.
\] 
Combining all of these estimates, we have the proposition.
\maybeqed\end{proof}

Taking expectations in~\eqref{eq:density-estimate} and then
comparing with the conformal covariance
rule~\eqref{confconv}, we get the following
corollary.

\begin{corollary}
The partition function for radial SLE from the interior 
is finite and equals
\[
\tmass_D(0,w) = |\psi'(0)|^\btilde \, |\psi'(w)|^b.
\]
Hence, it agrees with the partition function $\tmass_D(w,0)$
of ordinary radial SLE.
\end{corollary}

\subsection{Agreement with reversed radial SLE}
\label{sec:agreereversed}

\begin{proposition}[Domain Markov property]
\label{prop:dom-mark-prop}
Let $\gamma$ be radial SLE from the interior,
following the law $\mu_D^\#(0,w)$.
Let $\tau$ be a stopping time as in Definition~\ref{radfromint}.
Conditional on the starting segment $\gamma_\tau$, 
the remainder of $\gamma$ has the law of
crossing annulus SLE $\mu_{D\sm\gamma_\tau}^\#(\gamma(\tau),\infty)$
in the slit domain.
\end{proposition}

\begin{proof}
Let $T\ge\tau$ be a stopping time as in Definition~\ref{radfromint}.
By the domain Markov property for a whole-plane SLE curve $\gamma$
following the law $\mu_\C(0,\infty)$, 
conditional on~$\F_\tau$, 
the remainder of~$\gamma$ 
has the law of radial SLE 
$\mu^\#_{\C\sm\gamma_\tau}(\gamma(\tau),\infty)$
in the slit domain.

Conditional on~$\F_\tau$,
the density of $\mu_D(0,w)$ on $\gamma_T$ with
respect to whole-plane SLE $\mu_\C(0,\infty)$ 
(i.e., with respect to the measure 
$\mu^\#_{\C\sm\gamma_\tau}(\gamma(\tau),\infty)$) is
\begin{multline*}
\frac{d\mu_D(0,w)}
     {d\mu_\C(0,\infty)}(\gamma_T)
\biggm/
\frac{d\mu_D(0,w)}
     {d\mu_\C(0,\infty)}(\gamma_\tau)
\\ =
\exp\Bigl\{\frac\cc2
  \Lambda(\gamma_T,\partial D;\C\sm\gamma_\tau)\Bigr\}
\frac{\tmass_{D\sm\gamma_T}(\gamma(T),w)}
     {\tmass_{\C\sm\gamma_T}(\gamma(T),\infty)}
\biggm/
\frac{\tmass_{D\sm\gamma_\tau}(\gamma(\tau),w)}
     {\tmass_{\C\sm\gamma_\tau}(\gamma(\tau),\infty)}.
\end{multline*}
By Proposition~\ref{pfprop}, we know that
this quantity is the density of crossing annulus SLE 
$\mu_{D\sm\gamma_\tau}^\#(\gamma(\tau),w)$
with respect to radial SLE
$\mu_{\C\sm\gamma_\tau}^\#(\gamma(\tau),\infty)$,
up to time~$T$. 
\maybeqed\end{proof}

\begin{proposition}
\label{prop:cont-term-pt}
The definition of radial SLE from the interior gives a
random curve $\gamma:(0,T_D)\to D$ with $\gamma(0+)=0$ and
$\gamma(T_D-)=w$, where $T_D$ is the random time at which
$\gamma$ leaves $D$.
\end{proposition}

\begin{proof}
Let $t$ be small enough that $\gamma_t\subset D$ deterministically.
By Proposition~\ref{prop:dom-mark-prop}, conditional on $\gamma_t$,
the remainder of $\gamma$ is crossing annulus SLE in $D$ to $w$.
Because crossing annulus SLE is defined to be absolutely continuous
with respect to chordal SLE, it is continuous up to its terminal
point, and therefore $\gamma(T_D-)=w$.
\maybeqed\end{proof}

\begin{proposition}
\label{probrev}
As a probability measure, radial SLE from the interior
$\mu^\#_D(0,w)$ is the reversal of radial SLE $\mu^\#_D(w,0)$.
\end{proposition}

\begin{proof}
Let $\gamma$ be a radial SLE curve from the interior,
following the law $\mu_D^\#(0,w)$.
Conditional on the initial segment $\gamma_t$, the
remainder of the curve has the law of annulus SLE in the
slit domain, $\mu^\#_{D\sm\gamma_t}(\gamma(t),w)$.
Since annulus SLE is reversible, this is the reversal of
$\mu^\#_{D\sm\gamma_t}(w,\gamma(t))$.
By Corollary~\ref{cor:gamma-convergence}, the weak limit of
these measures as $t\to0$, uniformly in~$\gamma$, is the
reversal of $\mu^\#_{D}(w,0)$.
But the weak limit of the measures on post-$t$ segments of
$\gamma$ is $\mu_D^\#(0,w)$ itself, whence the result
follows.
\maybeqed\end{proof}

The non-probability measure for radial SLE from the
interior, $\mu_D(0,w)$, is therefore the reversal of radial
SLE $\mu_D(w,0)$, because the corresponding partition
functions and probability measures agree.
This concludes the proof of Theorem~\ref{slethm}.

\subsection{Multiply connected domains}
\label{sec:multconn}

If $D$ is not simply connected, we can still define
$\mu_D(0,w)$ by~\eqref{definition}, which may be regarded as
bulk/boundary SLE in the terminology of~\cite{L11}.  
For each curve~$\gamma$ from~$0$ to~$w$ in~$D$, 
we may find a simply connected $D_1\subset
D$  agreeing with~$D$ near~$w$.  
Now we use~\eqref{cascade} and~\eqref{loopdecomp} to observe that
\[
\frac{d\mu_D(0,w;D_1)}
     {d\mu_{D_1}(0,w)}(\gamma)
= \exp\Bigl\{ -\frac\cc2 \Lambda(\gamma,D_1^c;D) \Bigr\}.
\]
where $\mu_D(0,w;D_1)$ denotes $\mu_D(0,w)$ restricted to
curves that lie in $D_1$.  But boundary/bulk SLE is defined
in~\cite{L11} to satisfy this restriction rule. 
Moreover, we know that $\mu_{D_1}(0,w)$ is the
reversal of $\mu_{D_1}(w,0)$.  
We conclude that Corollary~\ref{boundarybulkcor} holds and
bulk/boundary SLE is conformally covariant.

\section{Normalizing the Brownian loop measure}
\label{sec:blm}

The aim of this section is to prove Theorem~\ref{main} and Proposition~\ref{good-lmstar-estimate}.

Invariance of
$\Lambda^*$ under M\"obius transformations
implies  that 
  the definition \eqref{feb28.1} does not
change if we shrink down at a point on the
Riemann sphere other than the origin.
  In other words, 
\begin{equation}  \label{feb28.1.B}
 \Lambda^*(V_1,V_2) = \lim_{r \downarrow 0}
             [\Lambda(V_1,V_2;\Outer_r(z)) - \log \log(1/r)],
\end{equation}
\begin{equation}  \label{feb28.1.C}
  \Lambda^*(V_1,V_2) = \lim_{R \rightarrow \infty}
             [\Lambda(V_1,V_2;\Disk_R) - \log \log R].
\end{equation}

Theorem
\ref{main1} establishes the existence
of the limit in \eqref{feb28.1}. Theorem \ref{main2}
proves the alternate forms \eqref{feb28.1.B}
 and \eqref{feb28.1.C}. If
$f$ is a M\"obius transformation, then conformal invariance
of the loop measure implies
\[   \Lambda(V_1,V_2;\Outer_r) = \Lambda(f(V_1),f(V_2);
f(\Outer_r)). \]
Invariance of $\Lambda^*$ under dilations, translations, and
inversions can be deduced from this and \eqref{feb28.1},
\eqref{feb28.1.B}, and \eqref{feb28.1.C}, respectively.

The proof of Theorem~\ref{main} really only uses
 standard arguments about planar Brownian
motion but we need to control the  error terms.
In  order to make the proof easier to understand,
 we have split it
  into three subsections.  The first subsection considers estimates
for planar Brownian motion.  Readers who are well acquainted
with planar Brownian motion may wish to skip this subsection and
refer back as necessary.  This subsection assumes
knowledge of planar Brownian motion as in 
\cite[Chapter 2]{Lconf}.
  The next subsection discusses the
Brownian (boundary)
 bubble measure and gives estimates for it.  The
Brownian loop measure is a measure on unrooted loops, but for
computational purposes it is often easier to associate to
each unrooted loop a particular rooted loop yielding an
expression in terms of Brownian bubbles.  The third subsection
proves the main theorem by giving estimates for the loop measure.
The last subsection proves Proposition~\ref{good-lmstar-estimate}.

\subsection{Lemmas about Brownian motion}  
\label{sec:blmlemmas}

 The exact form of the Poisson kernel in the unit
disk shows that there is
a $c$ such that for all $|z| \leq 1/2$ and $|w| = 1$,
\[    | 2\pi \, h_\Disk(z,w) - 1|  \leq c \, |z| . \]
By taking an inversion, we get that if $|z| \geq 2$, 
\begin{equation}  \label{feb28.7}
    |2\pi \, h_{\Outer}(z,w) -1| \leq
   \frac c{ |z|} . 
\end{equation}
 It is standard that 
\begin{equation} \label{logest}
         h_{A_R}( z, C_R) =
    \frac{\log |z|}{\log R}, \;\;\;\;
     1 < |z| < R. 
\end{equation}
In particular,
\begin{equation} \label{feb27.11}
    \exc_{A_R}(1,C_R) = \frac{1}{\log R} , \;\;\;\;
   \exc_{A_R}(R,C_1) = \frac{1}{R \, \log R}, 
\end{equation}

If $V \subset \p D$ is smooth, let $\overline h_D(z,w;V)
= h_D(z,w)/h_D(z,V)$ for $w \in V$.  In other words,
$\overline h_D(z,w;V)$ is the density of the exit
distribution  of
a Brownian motion {\em conditioned} so that it exits
at $V$.  We similarly define $\overline h_{\p D}
(z,w;V).$

\begin{lemma} There exists $c < \infty$ such that
the following holds.  Suppose $R > 0$ and  $D$ is
a domain with $A_R \subset D \subset \Outer$.  
    Then
\begin{equation}  \label{feb28.2.alt}
 |2 \pi  \, \overline h_{ D}(z,w;C_1) - 1|
     \leq c \, \frac{\log R}{R} ,\;\;\;\;
 |w| = 1, \,z \in D \cap \overline \Outer_{R/2}. 
\end{equation}
\end{lemma}

\begin{remark}  The conclusion of this lemma is very
reasonable.  If a Brownian motion starting at a point $z$
far from the origin exits $  D$ at $C_1$, then
the hitting distribution is almost uniform.  This uses
the fact that $  D \cap \Disk_R$ is the same
as~$A_R$.  The important result is the estimate  of
the error term.
\end{remark}

\begin{proof}  
Assume $|w| = 1$.  Let $\tau = \tau_{ D}$ and let
$\p^* = \p D \cap \Outer$.
  It suffices to prove the
estimate for $|z| = R/2$. 
  For every $|\zeta| \geq R/2$, \eqref{feb28.7} gives
\begin{equation}  \label{mar1.11}
  | 2\pi h_{\Outer}(\zeta,w) - 1 | \leq \frac c R.
\end{equation}
Note that
\[   h_{\Outer}(z,w) =  h_{ D}(z,w) + \E^z[h_\Outer(
B_\tau,w); B_\tau \in \p^*].\]
Using \eqref{mar1.11}, we get
\[  2\pi \E^z[h_\Outer(
B_\tau,w); B_\tau  \in \p^*] = 
h_{ D}(z,\p^*)  
 \,[1 + O(R^{-1})]. \]
Therefore,
\begin{equation}  \label{mar1.12}
  2\pi \,h_{ D}(z,w) = h_{  D}(z,C_1)
  + O(R^{-1}). 
\end{equation}
Since $h_{  D}(z,C_1)$ is bounded below
by the probability of reaching $C_1$ before $C_R$, 
\eqref{logest} implies
\[   h_{  D}(z,C_1)
              \geq \frac{\log 2}{\log R}, \]
and hence \eqref{mar1.12} implies
\begin{equation*}  \label{mar2.1}
 2 \pi \,  h_{  D}(z,w) =  h_{ D}(z,C_1)
  \, \left[ 1 + O\left( \frac {\log R} R \right)
  \right]. \qedhere
\end{equation*}
\end{proof}

\begin{corollary}  There exists $c < \infty$ such
that if $R \geq 2$, $|z| = 1$, $|w| = R$,   then
\begin{equation}  \label{mar2.5}
 \left|h_{\p A_R}(z,w) -\frac{1}{2 \pi R \, \log R}\right|
 \leq \frac{c}{R^2}. 
\end{equation}
\end{corollary}

\begin{proof}  Recall that $h_{\p A_r}(z,w) =
h_{\p A_r}(w,z)$.
We know from \eqref{feb27.11} that
\[   \int_{C_1} h_{\p A_R}(w,\zeta) \, |d\zeta| =
   \frac{1}{R \log R}. \]
Also, by definition,
\[         h_{\p A_R}(w,z) = \frac{\overline h_{\p A_R}
   (w,z)}{R \,\log R}. \]
Note that $\overline h_{\p A_R}
   (w,z)$ is bounded by the minimum and maximum values
of $\overline h_{A_R}(\hat w,z)$ over $|\hat w| = R/2$,
which by \eqref{feb28.2.alt} satisfy
\[  \overline h_{A_R}(\hat w,z) = \frac 1{2\pi}
  + O\left(\frac{\log R}{R}\right). \qedhere\]
\end{proof}

\begin{lemma} \label{feb28.lemma1}
Suppose $D$ is a nonpolar domain containing
$\Disk$. If $ 0 < s < 1$, let $D_s = D \cap \Outer_s$.
Then if $s < r \leq 1/2$ and $|z| = r$, 
\[    \frac{\log r}
  {\log s} \leq  h_{D_s}(z,C_s) \leq  \frac{\log r}{\log s}
 \, \left[1 - \frac{p\, \log 2}{(1-p) \log (1/r)}\right]^{-1} ,\]
where
\[  p =p_D =  \sup_{|\tilde w| = 1}  
        h_{D_{1/2}}(\tilde w,C_{1/2}) < 1. \]
\end{lemma}

\begin{proof}  
The inequality $p_D < 1$ follows immediately
from the fact that $D$ is nonpolar and contains $\Disk$.

Let $T = T_s =\inf\{t: B_t \in C_s \cup C_1\}$
and $\sigma = \sigma_{s,r}
=\inf\{t \geq T: B_t \in C_{r}\}$.  Then if $|z| \leq 1/2$,
\[  \Prob^z\{B_{\tau_{D_s}} \in C_s\} =
     \Prob^z\{B_T \in C_s\} +
    \Prob^z\{B_T \in C_1, B_{\tau_{D_s}} \in C_s\} .\]
By \eqref{logest}, 
\[   \Prob^z\{B_T \in C_s\}  = \frac{\log r}{\log s}, \]
which gives the lower bound.  Let
\begin{align*}
           q = q(r,s,D) &= \sup_{|\tilde z| = r}
    \Prob^{\tilde z} \left\{B_{\tau_{D_s}} \in C_s
\right\} , \\
          u = u(r,D) &= \sup_{|\tilde w| = 1}
\Prob^{\tilde w}\left \{B_{\tau_{D_r}} \in C_r
\right\}  .
\end{align*}
Then,
\begin{multline*}
\Prob^z\{ B_T \in C_1, B_{\tau_{D_s}} \in C_s\} \\
\leq
 \Prob^z\{ \sigma < \tau_{D_s} \mid B_T \in C_1\}
  \, \Prob^z\{  B_{\tau_{D_s}} \in C_s \mid 
  B_T \in C_1,  \sigma < \tau_{D_s} \}
 \leq  u \, q .
\end{multline*}
Applying this to the maximizing $\tilde z$, gives
\[     q  \leq   \frac{\log r}{\log s} + u \, q,\;\;\;\;\;
   q \leq  \frac{\log r}{(1-u) \,\log s} .\]

By \eqref{logest}, 
the probability that a Brownian motion starting on
$C_{1/2}$ reaches $C_r$ before reaching $C_1$ is
$\log 2/\log(1/r)$.  Using a similar argument as in
the previous paragraph, we see that
\begin{equation*}
\label{mar2.25} \tag{4.9 bis}\sbox{0}{\popQED}
\makeatletter
\gdef\alt@tag{%
        \global\let\alt@tag\@empty
        \vtop{\ialign{\hfil##\cr
                \tagform@{4.9 bis}\cr
                \qedsymbol\cr}}%
        \setbox\z@
      }%
\makeatother
   u \leq p \, \frac{\log 2}{\log(1/r)}
   +  p \, u, \;\;\;\;\; u \leq \frac{p}{1-p}
  \, \frac{\log 2}{\log(1/r)}.
\end{equation*}
\end{proof}

\begin{proposition}
Suppose $D$ is a nonpolar domain containing the origin.
Then there exists $c = c_D < \infty$ such that if 
$0 < r \leq 1/2$, 
 $D_r = D \cap
\Outer_r$,  and $z \in D, |z| \geq 1$, 
\begin{equation}  \label{feb28.2}
        h_{D_r}(z,C_r) \leq \frac{c}{\log(1/r)}. 
\end{equation}
Also,  if $|w|= r$,  
\begin{equation}  \label{feb28.2.alter}
h_{D_r}(z,w) \leq \frac{c}{r \, \log(1/r)}.
\end{equation}
\end{proposition} 

\begin{proof}
Find   $0 < \beta < 1/2$ such that $\p D \cap \Outer_{2\beta}$
is nonpolar.  It suffices
to prove \eqref{feb28.2} for $r < \beta$. Since
 $\p D \cap \Outer_{2\beta}$
is nonpolar, 
   there exists $q = q_{D,\beta} > 0$ such that
for every $|z| \geq 2\beta$, the probability that a Brownian
motion starting at $z$ leaves $D$ before reaching
$C_{\beta}$ is at least $q$.
 If $r < \beta$, the probability
that a Brownian motion starting at $C_\beta$ reaches
$C_r$ before reaching $C_{2\beta}$ is 
\[  p(r) = \log 2/\log(2\beta/r) \leq \frac{c_1}{\log
(1/r)} . \]
Let $Q(r) = \sup_{|z| \geq  2 \beta}  h_{D_r}(z,C_r)$.
Then arguing similarly to the previous proof,
we have
\[    Q(r) \leq (1-q) \, \left[p(r) + 
  [1-p(r)]\, Q(r)\right] \leq p(r) + (1-q) \, Q(r), \]
which yields $ Q(r) \leq p(r)/q$.  This gives
\eqref{feb28.2} and \eqref{feb28.2.alter} follows from
\[ h_{D_r}(z,w) \leq h_{D_{2r}}(z,C_{2r})
 \, \sup_{|\zeta| = 2r} \, h_{\Outer_r}(\zeta,w).\qedhere\]
\end{proof}

\begin{proposition}  There exists $c < \infty$
such that the following holds.
Suppose $|z|  = 1/2$  and 
$0 < s < r < 1/8$.  Let $D_{s,r} = \Outer_s \cap
  \Outer_r(z)$.
Then for $|w| \geq 1$,
\begin{equation}  \label{feb28.4}
 \left| \frac{\log(rs)}{\log r} h_{D_{s,r}}(w,C_s) 
- 1\right| \leq  \frac{c}{
  \log(1/r)}.
\end{equation}
\end{proposition}

\begin{proof}
Without loss of generality, we assume $z = 1/2$.  Let
$L$ denote the line $\{x+iy: x = 1/4\}.$ Let
$\tau = \tau_{D_{s,r}}$,  $T$
 the first time a Brownian motion  reaches
$C_r \cup C_r(z)$,  and $\sigma$ the first time after
$T$ that the Brownian motion returns to~$L$.
 By symmetry, for every $w \in L$,
\[        \Prob^w\{B_T \in C_r\} = \frac 12 . \]
Using Lemma \ref{feb28.lemma1}, we get
\[  \Prob^w\{ \tau  < \sigma
  \mid B_T \in C_r\} = \frac{\log r}{\log s}
  \, \left[1 + O\left(\frac 1 {\log(1/r)}\right)
  \right].\]
Therefore, for every $w \in L$,
\[   \Prob^w\{\tau < \sigma; B_\tau \in C_r(z)\}
   = \frac 12 , \]
\[    \Prob^w\{\tau < \sigma; B_\tau \in C_s\}
   = \frac{\log r}{2\log s}
  \, \left[1 + O\left(\frac 1 {\log(1/r)}\right)
  \right].\]
This establishes \eqref{feb28.4} for $w \in L$.
If $|w| \geq 1$, then the probability of reaching
$\Outer_r(z)$ before reaching $L$ is $O(1/\log(1/r))$
and the probability of reaching $\Outer_s$ before reaching
$L$ is $O(1/\log(1/s))$.  Using this we get
\eqref{feb28.4} for $|w| \geq 1$. 
\maybeqed\end{proof}

\begin{remark}
The end of the proof uses a well known fact.  Suppose one
performs independent trials with three possible outcomes
with probabilities $p,q,1-p-q$, respectively. Then the
probability that an outcome of the first type occurs
before one of the second type is $p/(p+q)$. 
\end{remark}

\subsection{Brownian bubble measure}
\label{sec:blmbubble}

Let $D$ be a nonpolar domain, 
$z\in\p D$ an analytic boundary point
and $\tilde D\subset D$ a domain that agrees with $D$ near $z$.
From the definition of the Brownian bubble measure, we see
that if $\p \tilde D \cap D$ is smooth
 \[ m(z;D,\tilde D) = 
          \pi \, \int_{\p  \tilde D \cap D}  
  h_{\p\tilde D}(z,w) \, h_{D}(w,z) \, |dw| . \]
This is
also equal to
$    \pi \,\p_{{\bf n}} f(z)  $
for the function $f(\zeta ) = h_D(\zeta,z) - h_{\tilde D}(
\zeta,z)$.  Let $h_{D,-}(V,z), h_{D,+}(V,z)$ denote the 
infimum and supremum, respectively, of $h_D(w,z)$ over
$w \in V$.  Then a simple estimate is
\begin{equation}   \label{feb27.10}
    h_{D,-}(\p \tilde D \cap D,z) \leq
  \frac{ m(z;D,\tilde D)}
 {\pi  \,  \exc_{\tilde D}(z, \p \tilde D \cap D)}
\leq
   h_{D,+}(\p \tilde D \cap D,z).
\end{equation}

\begin{lemma} \label{feb27.lemma2}
If $R > 1$, let
\[   \rho(R) = m(1;\Outer,A_R).\]
There exists $c < \infty$ such that
for all $R \geq 2$,
\[  \left| \rho(R)
     - \frac 1{2\log R}\right|
 \leq \frac{c}{R \, \log R} . \]
\end{lemma}

 \begin{remark} Rotational invariance implies
that $m(z;\Outer,A_R ) = \rho(R)$ for
all $|z| = 1.$
\end{remark}

\begin{proof}   By \eqref{feb27.11}, 
\[
    \exc(1,C_R; A_R)
    = \frac{1}{\log R}. 
\]
and by \eqref{feb28.7}, 
\[    2\pi \, h_{\Outer}(w,1) = 1 + O(R^{-1}) \]
for $w\in C_R$. 
We now use \eqref{feb27.10}.
 \maybeqed\end{proof}

The next lemma generalizes this to domains $ D$ with
$A_R \subset  D \subset \Outer$.  The result is similar
but the error term is a little larger.  Note that
the $q$ in the next lemma equals $1$ if $ D
= \Outer$.
 
\begin{lemma}  Suppose $R \geq 2$ and
$D$ is a domain satisfying
$A_R \subset D \subset \Outer$.   
Let $q$ be the probability that a Brownian motion
started uniformly on $C_R$ exits $D$ at
$C_1$, i.e.,
\[ q = q(R,D) = \frac{1}{2 \pi R} \int_{C_R}
      h_{ D}(z,C_1) \, |dz|.\]
Then if $|w| = 1$,
\begin{align}  
\label{mar2.20}
 m(w; D,A_R) &= \frac{q}{2\, \log R}
    \, \left[1 + O\left(\frac {\log R}{R} \right)
 \right], \\
\label{mar2.21}
  m(w;\Outer,D) &= \frac{1-q}{2 \, \log R}
   + 
  O\left (\frac {q \log R + 1}{ R \log R}
\right). 
\end{align}
\end{lemma}

\begin{proof} By definition,
\[ 
   m(w; D,A_R)   =   \pi \int_{C_R}
   h_{\p A_R}(w,z) \, h_{D}(z,w) \, |dz|. \]
For $z\in C_R$, by \eqref{mar2.5} we know that
\[   h_{\p A_R}(w,z) = \frac{1}{2 \pi \, R \,
 \log R} \, \left[1  + O\left(\frac
   {\log R}R\right)\right],\]
and by \eqref{feb28.2.alt} we know that
\[    h_{ D}(z,w) = \frac{1}{2\pi} \,
  h_{D}(z,C_1) \, 
\left[1  + O\left(\frac
   {\log R}R\right)\right].\]
Combining these gives \eqref{mar2.20},  and \eqref
{mar2.21} follows from Lemma
\ref{feb27.lemma2} and 
\[m (w;\Outer, A_R) =   m(w; D,A_R) + m(w;\Outer,D) 
.\qedhere\]
\end{proof}

\begin{corollary}  \label{mar1.lemma2}
There exists $c < \infty$
such that the following is true.
Suppose $R \geq 2$ and
$D$  is a 
domain with $A_R \subset D 
\subset \Outer$.   Suppose $\p D 
\cap \Outer_R$ is nonpolar and hence
\[  p = p_{R,D}  := \sup_{|z| = R} h_{D\cap \Outer
_{R/2}}(z,C_{R/2}) < 1. \]
Then, if $|w| = 1$, 
 \[ \left|
m(w;\Outer ,  D )
   - \frac{1}{2  \, \log R} \right| \leq  
\frac{c}
   {(1-p) \, \log^2R}.\]
\end{corollary}

\begin{proof}  Let $q$ be as in the previous
lemma. By \eqref{mar2.25} we see that
\[               q \leq \frac{p \, \log 2}{(1-p)
  \, \log R}. \]
and hence the result follows from \eqref{mar2.21}.
\maybeqed\end{proof}

\begin{remark}  We will used scaled versions of this corollary.
For example,
 if $D$ is a nonpolar domain containing $\Disk$,
$r < 1/2$,  $D_r = D \cap \Outer_r$, and $|w| = r$,
\[
\left|
r^2 \, m(w;\Outer_r ,  D_r )
   - \frac{1}{2  \, \log (1/r)} \right| \leq  
\frac{c}
   {(1-p) \log^2(1/r)}, 
\]
where
\[  p = \sup_{|z|=1} h_{D_{1/2}}(z,C_{1/2}).\]
\end{remark}

\begin{proposition}  \label{mar3.1}
 Suppose $V$ is a nonpolar
closed set,
$z \neq 0$, and $z,0 \notin  V$.
For $0 < r,s < \infty$, let
\[          
          D_{s,r} = \Outer_s \cap \Outer_r(z) . \]
Then   as $s,r \downarrow 0$, if $|w| = s$,
\[  \frac 1 \pi \, m(w;D_{s,r},D_{s,r}
\setminus V) = \frac{1}{2\pi s^2\,\log(1/s)
 } \,
 \frac{\log r}{\log(rs)} \, \left[1 + O\left(
  {\delta_{r,s}}\right) \right],\]
where $\delta_{r,s} = (\log (1/r))^{-1} + (\log (1/s))^{-1}.$
\end{proposition}

\begin{remark}
The implicit constants in the $O(\cdot)$ term depend
on~$V,z$ but not on~$w$.
\end{remark}

\begin{proof}
We will use \eqref{feb27.10} and write
$\delta = \delta_{r,s}$.
  By scaling we may assume $z=2$
and let $d= \min\{2,\dist(0,V),\dist(2,V)\}$.  We will only consider
$r,s \leq d/2$. By 
\eqref{feb27.11},
\[  \exc(w, C_d; A_{s,d}) = s^{-1} \,
\exc(w/s, C_{d/s}; A_{d/s}) = \frac{1}{ s \,\log(1/s)} \,
     \left[1+O(\delta)\right].\]
Using this and \eqref{logest} we can see that
\[    \exc(w, V; D_{s,r} \setminus V) = \frac{1}{ s\,\log(1/s)} \,
     \left[1+O(\delta)\right].\]
For $\zeta \in V$, \eqref{feb28.4} gives
\begin{equation} \label{mar3.2}
  h_{D_{s,r}}(\zeta,C_s) 
  = \frac{\log r}{\log(rs)} \, \left[1 + O(\delta)\right]. 
\end{equation}
 Therefore, by \eqref{feb28.2.alt},
 \[  h_{D_{s,r}}(\zeta,w) 
  = \frac{\log r}{2\pi s\log(rs)} \, 
    \left[1 + O(\delta)\right]. \qedhere\]
\end{proof}

The next proposition is the analogue of
Proposition \ref{mar3.1} with
$z = \infty$.

\begin{proposition}  \label{mar3.1.infty}
 Suppose $V$ is a nonpolar
compact set, 
with $0 \notin  V$.
For $0 < s,r < \infty$, let
\[          
          D_{s,r} = \Outer_s \cap \Disk_{1/r} . \]
Then, as $s,r \downarrow 0$, if $|w| = s$,
\[  \frac 1 \pi \, m(w;D_{s,r},D_{s,r}
\setminus V) = \frac{1}{2\pi s^2\,\log(1/s)
 } \,
 \frac{\log r}{\log(rs)} \, \left[1 + O\left(
  {\delta_{r,s}}\right) \right],\]
where $\delta_{r,s} = (\log (1/r))^{-1} + (\log (1/s))^{-1}.$
\end{proposition}

\begin{proof}  The proof is the same as for the previous
proposition.  In fact, it is slightly easier because
\eqref{mar3.2} is justified by \eqref{logest}.
\maybeqed\end{proof}

\subsection{Brownian loop measure}
\label{sec:blmloop}

\begin{lemma} \label{mar1.lem1}
 Suppose $V_1,V_2$ are closed sets and
$D$ is a domain. 
Let
\[      
         V^j  = \overline { A_{e^{j-1},e^{j}}}, \;\;\;\;
   \Outer^j = \Outer_{e^{j}}, \;\;\;\;
  D^j = D \cap \Outer^j.\]
Then
\begin{equation}  \label{mar2.40}
  \Lambda(V_1,V_2;D) = \sum_{j=-\infty}^\infty
    \Lambda(V_1,V_2,V^{j+1}; D^j) . 
\end{equation}
\end{lemma}

\begin{proof}  For each unrooted loop, consider the
point on the loop closest to the origin.  The measure
of the set of loops for which the distance to the
origin is exactly $e^j$ for some integer $j$ is $0$.
For
each loop, there is a unique $j$ such that the loop
is in $\Outer^j$ but not in $\Outer^{j+1}$.
Except for a set of loops of measure zero, 
such a loop intersects $V^{j+1}$ but does not intersect
$V^k$ for $k <j+1$, and hence each loop is counted
exactly once on the right-hand side of \eqref{mar2.40}.
   \maybeqed\end{proof}

\begin{lemma}  \label{feb25.lem1}
 There exists $c < \infty$
such that if $0 < s < 1 , R\geq 2$, 
\[   \left|\Lambda(C_1,C_R;\Outer_s)
 -  \log \left[\frac{\log(R/s)}{\log R}\right] 
  \right
| \leq   \frac c {R \, \log R}. \]
In particular, there exists $c < \infty$ such that
if $R \geq 2/s >4$,  
\[ \left|\Lambda(C_1,C_R;\Outer_{s})
-\frac{\log (1/s)}{\log R}\right|
   \leq \frac{c \log^2 (1/s)}{\log^2 R} .\]
\end{lemma}

\begin{proof}
By \eqref{feb27.3.new}, 
rotational invariance, and the scaling rule, we get
\[    \Lambda(C_1,C_R;\Outer_s)
 =
2 \int_{s}^1  r\, 
m(r; \Outer_{r},
    A_{r,R}) \, dr =
   2 \int_{s}^1  r^{-1} \, \rho(R/r) \, dr, \]
where $\rho$ is as in Lemma \ref{feb27.lemma2}.
From that lemma, we know that
\[  \rho(R/r) = \frac{1}{2\log(R/r)} + O\left(\frac{r}{
     R \log (R/r)}\right), \]
and hence
\[   \Lambda(C_1,C_R;\Outer_s)
=    O\left(\frac 1{R \log R}\right) +
 \int_s^1  \frac{1}{r \, (\log R - \log r)} \, dr
     .\]
The first assertion follows by integrating and the
second from the expansion
\[ \log \left[\frac{\log(R/s)}{\log R}\right] =
   \frac{\log(1/s)}{\log R} + O
\biggl(\frac{\log^2(1/s)}{\log^2R}\biggr). \qedhere\]
\end{proof}

\begin{lemma}  \label{mar1.lemma3}
Suppose $V$ is a closed, nonpolar set
with $0 \notin V$ and $\alpha > 1$.
 There exists $c = c_{V,\alpha} < \infty$
such that for sufficiently small $r$,
\[ \left|\Lambda(V,\Outer_r\setminus 
\Outer_{\alpha r}; \Outer_r) - \frac{\log \alpha}{\log(1/r)}
  \right| \leq \frac{c}{\log^2(1/r)}. \]
 \end{lemma}

\begin{proof}  By scaling, we may assume that
$\dist(0,V) = 1$.  It suffices to prove
the result for $r$ sufficiently small.
 By \eqref{feb27.3.new}, we have
\[  \Lambda(V,\Outer_r\setminus 
\Outer_{\alpha r}; \Outer_r) = \frac 1 \pi
  \int_0^{2\pi} \int_r^{\alpha r} m
  (s e^{i\theta};\Outer_s,D_s) \,s\, ds \, d\theta, \]
where $D_s = \Outer_s \setminus V$. 
By Corollary \ref{mar1.lemma2}, for $r \leq s \leq \alpha r$,
\[  m
  (s e^{i\theta};\Outer_s,D_s) = \frac{1}{2
 s^2 \, \log(1/s)} \, \left[1 + O
\left(\frac{1}{\log (1/r)}
  \right) \right].\]
Therefore,
\[  \Lambda(V,\Outer_r\setminus 
\Outer_{\alpha r}; \Outer_r) = 
\left[1 + O
\left(\frac{1}{\log (1/r)}
  \right) \right] \int_r^{\alpha r}
  \frac{ds}{s \,\log(1/s)}.\]
Also,
\[
\begin{aligned}
 \int_r^{\alpha r}
  \frac{ds}{s \,\log(1/s)} &=
 \log\log\left(\frac 1 r\right)
 - \log \log \left(\frac 1{\alpha r}\right)\\
 &= \frac{\log \alpha}{\log (1/r)} + O\left(
 \frac 1{\log^2(1/r)}\right).
\end{aligned}\tag*{\raisebox{-1.1\baselineskip}{\qedhere}}
\]
\end{proof}

\begin{lemma}  \label{feb27.newlem}
Suppose $V_1,V_2$ are
nonpolar
closed subsets of the Riemann sphere with $0
\notin V_1$.  Then there exists $c 
 = c_{V_1,V_2} < \infty$
such that for all $r \leq \dist(0,V_1)/2$,
\begin{equation}
  \label{feb27.4}
\Lambda(V_1, C_r; \C \setminus
  V_2) \leq \frac{c}{\log(1/r)}. 
\end{equation}
\end{lemma}

\begin{proof}  Constants in this proof depend on
$V_1,V_2$. Without loss of generality assume
$0 \notin V_2$ and let $D_r = \Outer_r \setminus
V_2$.  We will first prove the result
for $r \leq r_0 = [\dist(0,V_1) \wedge \dist(0,V_2)]/2$. 
By \eqref{feb27.3.new}, we have 
\begin{equation}
\Lambda(V_1, C_r; \C \setminus
  V_2)  = \frac 1 \pi \int_{|z| \leq r}
   m(z;D_{|z|}, D_{|z|} \setminus
 V_1) \, dA(z) . 
\end{equation}
By \eqref{feb28.2.alter},
\[        h_{D_r}(w,z) \leq \frac{c} {
   r\, \log(1/r)}, \;\;\;\;\;  w \in V_1,
  \;\;\; |z| = r. \]
By comparison with an annulus, we get
\[   \exc_{D_{r}\setminus
 V_1}(z,V_1) \leq \frac{c}{r \, \log
 (1/r)}, \;\;\;\; |z| = r.\]
Using \eqref{feb27.10},
we then have 
\[  \frac 1 \pi \, m(z;D_{|z|}, D_{|z|} \setminus
 V_1) \leq \frac{c}{|z|^2 \, \log^2 (1/|z|)} .\]
By integrating, we get \eqref{feb27.4} for $r \leq r_0$.

 Let $r_1 = \dist(0,V_1)/2$ 
and  note that
\[ \Lambda(V_1, C_{r_1}; \C \setminus
  V_2)  = \Lambda(V_1, C_{r_0}; \C \setminus
  V_2)  + \Lambda(V_1, C_{r_1}; \Outer_{r_0} \setminus
  V_2 ) .\]
Using Lemma \ref{diamlem} we can see that
$ \Lambda(V_1, C_{r_1}; \Outer_{r_0} \setminus
  V_2 ) < \infty $.
Therefore,
\[  \Lambda(V_1, C_{r_1}; \C \setminus
  V_2)  < \infty , \]
and we can conclude \eqref{feb27.4} for $r_0 \leq r \leq
r_1$ with  a different constant.
\maybeqed\end{proof}

\begin{corollary}  \label{easycor}
Suppose $V_1,V_2$ are disjoint closed subsets of
the Riemann sphere and $D$ is a nonpolar domain.
Then
\[    \Lambda(V_1,V_2;D) < \infty. \]
\end{corollary}

\begin{proof}  Assume $0 \notin V_1$.
  Lemma
\ref{feb27.newlem} shows that 
  $\Lambda(V_1,
\overline \Disk_s; D) < \infty$ for some
$s > 0$.
Note that
\[  \Lambda(V_1,V_2;D) \leq
 \Lambda(V_1,
\overline \Disk_s; D) + \Lambda(V_1,V_2;\Outer_s) . \]
Since at least one of $V_1,V_2$ is compact, 
 Lemma \ref{diamlem}
implies that
\[   \Lambda(V_1,V_2;\Outer_s) < \infty.\qedhere\]
\end{proof}

\begin{theorem}  \label{main1}
 Suppose $V_1,V_2$ are disjoint,
nonpolar
closed subsets of the Riemann sphere.
 Then the limit
\begin{equation}  \label{feb27.1}
  \Lambda^*(V_1,V_2) = \lim_{r \downarrow 0}
    \left[\Lambda(V_1,V_2;\Outer_r) - \log \log(1/r)
\right] 
\end{equation}
exists.
\end{theorem}

\begin{proof} 
Without loss of generality, assume that $\dist(0,V_1)
\geq 2$ and let $\Outer^k = \Outer_{e^{-k}}$. Let
$\hat V_2 \subset V_2$ be a nonpolar closed subset
with $0 \notin \hat V_2$. 
Constants in the proof depend on $V_1,V_2$.
Since $\Lambda(V_1,V_2;\Outer_r) $ increases as $r$
decreases to $0$, it suffices to establish the limit
\[ \lim_{k \rightarrow \infty}
    \left[\Lambda(V_1,V_2;\Outer^k) - \log  k
\right].\]
Repeated application of \eqref{cascade}
shows that if
 $k \geq 1$,
\[   \Lambda(V_1,V_2;\Outer^k)
  = \Lambda(V_1,V_2;\Outer^0) + \sum_{j=1}^k
       \Lambda(V_1,V_2, \Outer^{j-1}
 \setminus \Outer^j; \Outer^j).\]
Similarly, for fixed $k$, \eqref{cascade} implies
\begin{align*}
\lefteqn{\Lambda(V_1, \Outer^{k-1}
 \setminus \Outer^k; \Outer^k)
- \Lambda(V_1,V_2, \Outer^{k-1}
 \setminus \Outer^k; \Outer^k)}\hspace{14em}\\
& = \Lambda(V_1 , \Outer^{k-1}
 \setminus \Outer^k; \Outer^k \setminus V_2)\\
 & \leq  \Lambda(V_1 , \Outer^{k-1}
 \setminus \Outer^k; \Outer^k \setminus \hat V_2).
\end{align*}

From Lemma \ref{mar1.lemma3},
we can see that
 \[ \Lambda(V_1, \Outer^{k-1}
 \setminus \Outer^k; \Outer^k) = \frac 1 k +
  O_{V_1}\left(\frac 1 {k^2}\right), \]
and hence
the limit
\[   \lim_{k \rightarrow \infty}
   \biggl[-\log k + \sum_{j=1}^k
       \Lambda(V_1, \Outer^{k-1}
 \setminus \Outer^k; \Outer^k)\biggr]  \]
exists and is finite.
By
Lemma \ref{feb27.newlem}, we see that 
\[   \sum_{j=k}^\infty 
\Lambda(V_1 , \Outer^{j-1}
 \setminus \Outer^j; \Outer^j \setminus \hat V_2)
  = \Lambda(V_1,\overline \Disk^k; 
\C \setminus \hat V_2) 
 \leq \frac{c}{k}, \]
and hence
\[  \sum_{j=k}^\infty 
\left[\Lambda(V_1, \Outer^{j-1}
 \setminus \Outer^j; \Outer^j)
- \Lambda(V_1,V_2, \Outer^{j-1}
 \setminus \Outer^j; \Outer^j)
\right]     \leq \frac c k,           \]
where the constant $c$ depends on $V_1$
and $\hat V_2$ but not otherwise on $V_2$.
\maybeqed\end{proof}

\begin{remark}
 It follows from the proof that  
\[  \Lambda^*(V_1,V_2) = 
   \Lambda(V_1,V_2;\Outer^k) - \log k +
   O\left(\frac{1}{k}\right), \]
where the $O(\cdot)$ term depends on $V_1$ and
$\hat V_2$ but not otherwise on $V_2$. 
 As
a consequence we can see that if $0 \notin
V_1$ and  $V_{2,r} = V_2 \cap
\{|z| \geq r\}$, then
\begin{equation}  \label{cont}
\lim_{r \downarrow 0} \Lambda^*(V_1,V_{2,r})
    = \Lambda^*(V_1,V_2).
\end{equation}
\end{remark}

The definition of $\Lambda^*$ in \eqref{feb27.1} seems
to make the origin a special point.  Theorem
\ref{main2}
shows that this is not the case.

\begin{lemma}  \label{mar3.lemma4}
Suppose $V$ is a nonpolar closed set,
$z \neq 0$ and $0 \notin V$.  Let $\alpha > 0$.
There exist $c,r_0$ (depending on $z,V,\alpha$) such
that if $0 < r < r_0$, 
\[ \left| \Lambda(V,\C \setminus \Outer_r; \Outer_{
\alpha r}(z))
    - \log 2\right| \leq \frac{c}{\log(1/r)}. \]
\end{lemma}

\begin{proof}  We will first assume
$z \notin V$.
 For $s \leq r$, let $D_s = \Outer_s
 \cap \Outer_{\alpha r}(z)$. 
 As in \eqref{feb27.3.new},
\[ \Lambda(V, \C \setminus \Outer_r; \Outer_{\alpha r}(z))
= \frac 1 \pi \int_{|w| \leq r}
   m(w;D_{|w|},D_{|w|} \setminus V) \, dA(w) . \]
 By Proposition~\ref{mar3.1}, if $|w| =s \leq r$, 
\[ \frac 1 \pi \, m(w;D_{s},D_{s}
\setminus V) = \frac{1}{2\pi s^2\,\log(1/s)
 } \,
 \frac{\log r}{\log(rs)} \, \left[1 + O\left(
  \frac{1}{\log(1/r)}\right) \right],\]
and therefore,
\[ \Lambda(V, \C \setminus \Outer_r; \Outer_{\alpha r}(z))
= \log r \int_0^r \frac{ds}{s \, \log(1/s) \, \log(rs)}
 \,  \left[1 + O\left(
  \frac{1}{\log(1/r)}\right) \right].\]
A straightforward computation gives
\[  \log r \int_0^r \frac{ds}{s \, \log(1/s) \, \log(rs)}
=\log 2.\]
This finishes the proof for $z \notin V$.

If $z \in V$, let $V_1 \subset V$ be
a closed nonpolar set with $z \notin V_1$.  Then
\eqref{cascade} implies
\begin{multline*} 
\Lambda(V,\C \setminus \Outer_r; \Outer_{
\alpha r}(z)) \\
= \Lambda(V_1, \C \setminus \Outer_r;
   \Outer_{\alpha r}(z)) + \Lambda(V \setminus
  V_1, \C \setminus \Outer_r;
   \Outer_{\alpha r}(z) \setminus V_1) .
\end{multline*}
Since the previous paragraph applies to $V_1$ it suffices
to show that
\[ \Lambda(V \setminus
  V_1, \C \setminus \Outer_r;
   \Outer_{\alpha r}(z) \setminus V_1) = O\left(
  \frac 1 {\log (1/r)}\right). \]
We can write
 \[ \Lambda(V\setminus V_1
, \C \setminus \Outer_r; \Outer_{\alpha r}(z)\setminus V_1)
= \frac 1 \pi \int_{|w| \leq r}
   m(w;D_s\setminus V_1,D_s \setminus V) \, dA(w) . \]
By using \eqref{mar2.20} and~\eqref{feb28.2} we can see that
\[   m(w;D_s\setminus V_1,D_s \setminus V) 
  \leq  \frac{c}{s^2\,\log^2(1/s)} , \]
and hence
\[ 
\Lambda(V\setminus V_1
, \C \setminus \Outer_r; \Outer_{\alpha r}(z)\setminus V_1)
  \leq   c \int_0^r \frac{ds}{s \, \log^2(1/s)} 
  \leq   \frac{c}{\log(1/r)}.\qedhere
\]
\end{proof}

The following is the equivalent lemma for $z = \infty$.  It
can be proved similarly or by conformal transformation.

\begin{lemma}  Suppose $V$ is a nonpolar closed set,
 and $0 \notin V$.  Let $\alpha > 0$.
There exists $c,r_0$ (depending on $V,\alpha$) such
that if $0 < r < r_0$, 
\[ \left| \Lambda(V,\C \setminus \Outer_r; \Disk_{\alpha / r})
    - \log 2\right| \leq \frac{c}{\log(1/r)}. \]
\end{lemma}

We extend this to $k$ closed sets.  

\begin{lemma} \label{klemma}
 Suppose $V_1,\ldots,V_k$ are closed
nonpolar subsets of $\C$ that do not contain~$0$.
 Let $z \ne 0$ and $\alpha > 0$.
There exist $c,r_0$ (depending on $z,\alpha,V_1,
\ldots,V_k$) such
that if $0 < r < r_0$, 
\begin{align*}
 \left| \Lambda(V_1,\ldots,V_k
,\C \setminus \Outer_r; \Outer_{
\alpha r}(z))
    - \log 2\right| &\leq \frac{c}{\log(1/r)}, \\
 \left| \Lambda(V_1,\ldots,V_k
,\C \setminus \Outer_r; \Disk_{\alpha/r})
    - \log 2\right| &\leq \frac{c}{\log(1/r)}. 
\end{align*}
\end{lemma}

\begin{proof}  If $k=2$, inclusion-exclusion
implies
\begin{multline*}
 \Lambda(V_1 \cup V_2
,\C \setminus \Outer_r; \Outer_{
\alpha r}(z)) + \Lambda(V_1,V_2
,\C \setminus \Outer_r; \Outer_{
\alpha r}(z))\\
  = \Lambda(V_1 
,\C \setminus \Outer_r; \Outer_{
\alpha r}(z)) + \Lambda(V_2
,\C \setminus \Outer_r; \Outer_{
\alpha r}(z))  .
\end{multline*}
Since Lemma \ref{mar3.lemma4} applies to $V_1 \cup V_2,
V_1,V_2$, we get the result.  The cases $k > 2$
and  $z=\infty$
are done similarly.
\maybeqed\end{proof}

\begin{theorem}  \label{main2}
Suppose $V_1,V_2$ are disjoint,
nonpolar
closed subsets of the Riemann sphere and $z \in
\C$.
 Then 
\[ \Lambda^*(V_1,V_2) = \lim_{r \downarrow 0}
    \left[\Lambda(V_1,V_2;\Outer_r(z)) - \log \log(1/r)
\right]  .
\]
Moreover,
\[ \Lambda^*(V_1,V_2) = \lim_{R \rightarrow \infty}
    \left[\Lambda(V_1,V_2;\Disk_R) - \log \log R
\right]  .\]

\end{theorem}

\begin{proof}   We will assume $0 \notin V_1$.
Using \eqref{feb27.1}, we see that
it suffices to prove that
\[ 
\lim_{r \downarrow 0}
    \left[\Lambda(V_1,V_2;\Outer_r(z)) - \Lambda(V_1,
V_2;\Outer_r)
\right]  = 0 .\]
Note that
\begin{multline*}
 \Lambda(V_1,V_2;\Outer_r(z)) - \Lambda(V_1,
V_2;\Outer_r) \\
= \Lambda(V_1,V_2,\C \setminus \Outer_r; \Outer_r(z)) -
    \Lambda(V_1,V_2,\C \setminus \Outer_r(z); \Outer_r).
\end{multline*}
Lemma \ref{klemma} implies
\begin{equation}  \label{feb27.2}
   \Lambda(V_1,V_2,\C \setminus \Outer_r; \Outer_r(z))
   = \log 2 + O\left(\frac 1 {\log(1/r)}\right), 
\end{equation}
where the constants in the error term depend on $z,V_1,V_2$.
Similarly, using translation invariance of the loop
measure, we can
see that  
\[ \Lambda(V_1,V_2,\C \setminus \Outer_r(z); \Outer_r)
   = \log 2 + O\left(\frac 1 {\log(1/r)}\right).
\]
The case $z=\infty$ is done similarly.
\maybeqed\end{proof}

If $V_1,V_2,\ldots,V_k$ are pairwise disjoint
nonpolar closed subsets
of the Riemann sphere,
we define
similarly
\[  \Lambda^*(V_1,\ldots,V_k) =
  \lim_{r \downarrow 0}
    \left[\Lambda(V_1,V_2,\ldots,
V_k;\Outer_r) - \log \log(1/r)
\right] .\]
One can prove the existence of the limit in the same
way or we can use the relation
\begin{equation}
\label{loopdecomp}
  \Lambda^*(V_1,\ldots,V_k)
  = \Lambda^*(V_1,\ldots,V_{k+1})
 + \Lambda(V_1,\ldots,V_k; \C \setminus V_{k+1}).
\end{equation}

\subsection{Estimate on the loops that cross an annulus}
\label{subs:lmstar-est}

We conclude the discussion of the normalized loop measure
by supplying the proof of
Proposition~\ref{good-lmstar-estimate}, 
which was stated on page~\pageref{good-lmstar-estimate}.

\begin{proof}
Since $\p D \cap \Disk = \emptyset$,  if  $R > 1$,
\[   \Lambda(K,\p D; \Disk_R) = \Lambda(K,C_1;\Disk_R)
  - \Lambda (K, C_1; D \cap \Disk_R). \]
  Taking limits as $R \rightarrow \infty$,
\[   \Lambda^*(K,\p D) = \lim_{R \rightarrow \infty}
    \left[\Lambda(K,C_1;\Disk_R) - \log \log R\right]
      - \Lambda(K,C_1;D ). \]
 Hence it suffices to show that
 \begin{equation}  \label{mar13.1}
  \lim_{R \rightarrow \infty}
    \left[\Lambda(K,C_1;\Disk_R) - \log \log R\right]
     = -\log \log (1/t) + O(t)
     \end{equation}
     and
     \begin{equation} \label{mar13.2ish}
     \Lambda(K,C_1;D) 
     = \log\left(1+\frac{\log\psi'(0)}{\log t}\right)
  + O(t).
  \end{equation}

Suppose $t < 1/8$ and $K \in \hulls_t $. Then,
by~\eqref{furthestroot}, if $R>1$,
\[ 
\Lambda(K,C_1;\D_R) 
= 
\frac1\pi
\int_0^{2\pi} \int_{1}^R m(re^{i\theta}; \D_r,\D_r\sm K)
\, r \, dr\, d\theta .\]
If $z \in K$, $r \geq 1$, $\theta \in [0,2\pi]$, then
\[h_{\Disk_r}(z,re^{i\theta}) = \frac{1}{2\pi r} \, \left[1 + O(t/r)
 \right], \]
 and hence
 \[  \frac{1}{\pi} \,  m(re^{i\theta}; \D_r,\D_r\sm K)
   = \exc_{\Disk_r\setminus K}(re^{i\theta},K) \,  \frac{1}{2 \pi r} 
   \, \left[1 + O(t/r)
 \right].\]
 Lemma \ref{exc-est} and conformal covariance gives
 \[    \exc_{\Disk_r\sm K}(re^{i\theta},K) =
 r^{-1} \, \exc_{\Disk\sm (r^{-1} K)}(e^{i\theta},r^{-1} K)=
    \frac{1}{r\log(r/t)} + O(t/r^2).\]
 Therefore,
 \[ \frac r\pi
\int_0^{2\pi}  m(re^{i\theta}; \D_r,\D_r\sm K)
\,  d\theta  =   \frac{1}{r\log(r/t)}   + O(t/r^2) \]
and
\[
 \Lambda(K,C_1;\D_R) 
= \int_1^R \left[ \frac{1}{r\log(r/t)}   + O(t/r^2)\right] \, dr
=\log \biggl[\frac{\log(R/t)}{\log(1/t)}\biggr] + O(t).
\]
This gives \eqref{mar13.1}.

 Let $D_r$ denote the connected component of $D \cap \Disk_r$
 containing the origin.  Then using \eqref{furthestroot} again, we
 get
 \[ 
\Lambda(K,C_1;D) 
= 
\frac1\pi
\int_0^{2\pi} \int_{1}^\infty m(re^{i\theta}; D_r,D_r\sm K)
\, r \, dr\, d\theta .\]
Our first claim is
\[   \Lambda(K,C_1;D) = \Lambda(C_t,C_1;D) \, [1 + O(t)].\]
In fact, for every $r \geq 1$,
\[  m(re^{i\theta}; D_r,D_r\sm K)=  m(re^{i\theta}; D_r,D_r\sm 
\overline \Disk_t)\,[1 + O(t)].\]
Indeed, this estimate follows from the two estimates
\[    h_{D_r}(z,re^{i\theta}) = h_{D_r}(0,re^{i\theta})
 \, [ 1 +O(t)], \;\;\;\; |z| \leq 4t,\]
 \[  \exc_{D_r\sm K}(re^{i\theta},K) =
   \exc_{D_r \sm \overline \D_t}(re^{i\theta},\overline{\D}_t)
    \, [ 1 +O(t)].\]
    The first follows from the fact that
  $h_{D_r}(\cdot,re^{i\theta})$ is a positive harmonic function
  on $\Disk$.  
  The second may be found in Lemma~\ref{L2.8Str}.

  To compute $\Lambda(C_t,C_1;D)$ we use 
  \eqref{march1.1} to write
\[ \Lambda(C_t,C_1;D)=   \frac 1 \pi \int_0^{2\pi}
   \int_0^t m_{\Outer_r}(re^{i\theta};D,\Disk) \, r \, dr\,
  d \theta .
\] 
Also, for $0<r<t$,
\begin{align*} 
\frac 1 \pi \, m_{\Outer_r}(re^{i\theta};D,\Disk) & =
   \int_{C_1}
    h_{\p A_{r,1}} (re^{i\theta},w)
    \, h_{D\cap\Outer_r}(w,re^{i\theta}) 
     \, |dw|\\
     & = \frac 1r\left[\frac{\delta_r}{2\pi  } +O(r)\right] \int_0^{2\pi}
       h_{D\cap\Outer_r}(e^{iy},re^{i\theta}) \, dy 
       \end{align*}
by~\eqref{mar2.5}.
Let $q_D(w,r)$ denote the probability that a Brownian
motion starting at $w$ hits $C_r$ before leaving $D$.
Then~\eqref{feb28.2.alt} implies that for $|w| = 1$,
 \[   h_{D\cap\Outer_r}(w,re^{i\theta})  =
            \frac{ q_D(w,r)}{2\pi r} \, [1 + O(r/\delta_r)].\]
 Therefore,
\[ r\int_0^{2\pi}
       h_{D\cap\Outer_r}(e^{iy},re^{i\theta}) \, dy  =
  \E[ q_D(B_\tau,r) ] \, [1 + O(r/\delta_r)], \]
  \[ \frac r \pi\, m_{\Outer_r}(re^{i\theta};D,\Disk)  = \frac{\delta_r}
  {2\pi r}\,  \E[ q_D(B_\tau,r) ] \, [1 + O(r/\delta_r)], \]  
   where $B_t$ is a Brownian motion 
   started at 0
   and $\tau$ is the first
   time $t$ with $|B_t| = 1$. 
   Let  $\psi:D \rightarrow \D$ be the unique conformal
   transformation with $\psi(0) = 0$, $\psi'(0) > 0$. 
   We have
\[    \Disk_{\psi'(0) \, r - O(r^2)}
  \subset \psi(\Disk_r)  \subset \Disk_{\psi'(0) \, r+ O(r^2)}.\]
  If $D \in \domain$, then $|\psi(w)| \geq 1/16$ for $|w| = 1$, and
  hence
  \[     q_D(w,r) = \frac{\log |\psi(w)|}{\log[\psi'(0)\,r + O(r^2)]}
   = - \delta_{\psi'(0)\,r} \, [\log |\psi(w)|] \, \left[1 + O(r\delta_r)\right].\]
   Hence
\[   \E[ q_D(B_\tau,r) ]  = - \delta_{\psi'(0)\,r} \, \E[\log |\psi(B_\tau)|]
 \, \left[1 + O(r\delta_r)\right].\]
By considering the harmonic function $H(z) = \log|\psi(z)/z|$, we
see that
\[  \E[\log |\psi(B_\tau)|] = \log \psi'(0).\]
    Hence,
  \[\frac r \pi\int_0^{2\pi}  m_{\Outer_r}(re^{i\theta},D,\Disk)\, d\theta
   = -\frac{\delta_r\,\delta_{\psi'(0)\,r}}{r}  \, [\log \psi'(0)]\, [1 + O(r/\delta_r)].\]
Also,
\[  \int_0^t \frac{\delta_r\,\delta_{\psi'(0)\,r}}{r} \, dr
= -\frac1{\log\psi'(0)}\log\left(1+\frac{\log\psi'(0)}{\log t}\right) , \]
\[   0\le\int_0^t  \delta_{\psi'(0)\,r} \,dr \le t\,\delta_{\psi'(0)\,t}. \]
This gives~\eqref{mar13.2ish}.
\maybeqed\end{proof}

\end{document}